\documentclass[11pt]{amsart}
\usepackage{amssymb, amsmath, amsfonts, amsthm, mathtools,hyperref} 
\usepackage{latexsym}
\usepackage[usenames,dvipsnames]{color}

\newtheorem{thm}{Theorem}[section] 
\newtheorem{dfn}[thm]{Definition}
\newtheorem{rmk}[thm]{Remark}
\newtheorem{rmks}[thm]{Remarks}
\newtheorem{cor}[thm]{Corollary}
\newtheorem{prop}[thm]{Proposition}

\newtheorem{lem}[thm]{Lemma}

\def\cyclic{\mathop{\kern0.9ex{{+}
\kern-2.2ex\raise-.28ex\hbox{\Large\hbox
{$\circlearrowright$}}}}\limits}

\def\buildrel#1_#2^#3{\mathrel{\mathop{\kern 0pt#1}\limits_{#2}^{#3}}}

\newcommand{\Id}{{\rm Id}}
\newcommand{\id}{{\rm id}}

\newcommand{\Ad}{{\rm Ad}}

\newcommand{\eps}{\varepsilon}

\newcommand{\bO}{{\bf\Omega}}

\newcommand{\C}{\mathbb C}

\newcommand{\bD}{{\bf D}}

\newcommand{\Z}{\mathbb Z} 

\newcommand{\R}{\mathbb R}
\newcommand{\K}{\mathbb K}
\newcommand{\Q}{\mathbb Q} 
 
\newcommand{\N}{\mathbb N}

\newcommand{\Tr}{\mbox{\rm Tr}}

\renewcommand{\k}{{\mathfrak{k}}{}}

\renewcommand{\k}{{\mathfrak{k}}{}}
\newcommand{\CO}{{\mathcal O}{}}

\newcommand{\CS}{\mathcal S}
\newcommand{\CE}{\mathcal E}
\newcommand{\CK}{\mathcal K}

\newcommand{\CD}{\mathcal D}

\newcommand{\CH}{\mathcal H}
\newcommand{\CB}{\mathcal B}

\newcommand{\CF}{\mathcal F}

\newcommand{\CG}{\mathcal G}
\newcommand{\CA}{\mathcal A}

\newcommand{\vf}{\varphi}

\renewcommand{\k}{{\bf k}}

\def\cref#1{Corollary~\ref{#1}}

\usepackage{fancyhdr}
\title{Deformation Quantization for actions of $\Q_p^{d}$}
\author{Victor Gayral}
\address{Laboratoire de Math\'ematiques\\
Universit\'e de Reims Champagne-Ardenne\\
Moulin de la Housse-BP 1039, 51687 Reims, France\\
e-mail: victor.gayral@univ-reims.fr}
\author{David Jondreville}
\address{Laboratoire de Math\'ematiques\\
Universit\'e de Reims Champagne-Ardenne\\
Moulin de la Housse-BP 1039, 51687 Reims, France\\
e-mail: david.jondreville@etudiant.univ-reims.fr}
\date{}
\begin{document}
\maketitle
\begin{abstract}
The main objective of this article is to develop the theory of
deformation of $C^*$-algebras endowed with  a group action,
from the perspective of non-formal equivariant quantization.
This program, initiated  in \cite{BG}, aims to 
 extend  Rieffel's deformation theory \cite{Ri} for more general groups than 
 $\mathbb R^d$.  In \cite{BG}, we have constructed such a theory for a class
 of non-Abelian Lie groups. In the present article, we study the somehow  opposite situation of Abelian but non-Lie
groups.  More specifically, we construct  here a deformation theory of $C^*$-algebras
endowed with
an action of a finite dimensional vector space over a non-Archimedean local field of characteristic 
different from 2. At the root of our construction stands the $p$-adic version of the Weyl quantization introduced by Haran \cite{Haran} and further extended by
 Bechata \cite{Bechata} and Unterberger \cite{Unterberger}.

\end{abstract}

{\Small{\bf Keywords:} Deformation
of $C^*$-algebras, Equivariant quantization,
Local fields, $p$-Adic pseudo-differential analysis}

\tableofcontents

\section{Introduction}
When formulated in the setting of operator algebras,  equivariant quantization  interconnects  both with deformation theory and
with quantum groups. These interconnections originate in the work of
Rieffel  \cite{Ri}, 
where it is shown that  Weyl's pseudo-differential calculus can be used to design a deformation theory for
$C^*$-algebras equipped with a continuous action of $\R^d$. Applying this deformation process  to $C_0(G)$, where $G$ is a  locally compact group possessing a copy 
of $\R^d$ as a closed subgroup and for the action $\rho\otimes\lambda$ of $G\times G$, in \cite{Ri3} Rieffel was further able to produce a large class of examples
of quantum groups in the $C^*$-algebraic setting. 
In \cite{BG}, Bieliavsky and one  of us   have successfully extended Rieffel's deformation theory
for actions
of negatively curved K\"ahlerian Lie groups on $C^*$-algebras. This was the first explicit example of a 
deformation theory for $C^*$-algebras coming from actions of non-Abelian groups and it was
based in an essential way on a generalization
(to all negatively curved K\"ahlerian Lie groups) of the $ax+b$-equivariant quantization
 due to Unterberger \cite{Un84}. (See also \cite{BGT}  for an extension of Rieffel's
 to construction to actions of the Heisenberg supergroup.)

There is another approach to quantization, due to
 Landstad and Raeburn \cite{L,L2,LR,LR2}, which also connects to quantum groups. At the conceptual 
 level, the starting point there
is that the twisted group $C^*$-algebra associated with a unitary $2$-cocycle should be considered as a quantization of the virtual dual group. This  approach to 
quantization has  been further developed  by Kasprzak in \cite{K1} to design a deformation theory for
 $C^*$-algebras endowed with a continuous action of a locally compact Abelian group,
from a unitary $2$-cocycle on the dual group. 
It was then observed by Bhowmick,  Neshveyev and 
Sangha in \cite{Nes-Shanga} that Kasprzak's construction still makes sense for actions of non-Abelian
locally compact groups, provided that the unitary $2$-cocycle is now chosen
in the dual quantum group (i.e$.$ the group von Neumann algebra).  An important point is that
unless the group used to deform is Abelian, the symmetries of the deformed objects  are now 
 given by a quantum group. All this suggests that quantum groups are naturally present in the context of 
 equivariant 
quantizations and in the associated deformation theories.

Very recently, Neshveyev and Tuset gave 
in \cite{NT} a great clarification of the role of quantum groups in deformations, by providing
a beautiful  theory holding with the  most imaginable degree of generality, namely for
   continuous actions of  locally compact 
quantum groups on $C^*$-algebras and from a
unitary $2$-cocycle on the dual quantum group. 
Their starting point is the work of De Commer  \cite{dC}, which shows that 
given a locally compact quantum group $(G,\Delta)$ (in the von Neumann algebraic setting 
\cite{KV1,KV2})  together with a dual measurable 
unitary $2$-cocycle\footnote{A dual measurable 
unitary $2$-cocycle $F$ on $G$, is
an unitary element of $L^\infty(\hat G)\bar\otimes L^\infty(\hat G)$ which satisfies the cocycle
condition 
$(F\otimes 1)(\hat\Delta\otimes\Id)(F)=(1\otimes F)(\Id\otimes\hat\Delta)(F)$.}
$F$ on $( G,\Delta)$,
 the pair $(\hat G,F\hat\Delta(.)F^*)$ is again a 
locally compact quantum group. The dual quantum group, denoted by $(G_F,\Delta)$, is thought as
 the deformation of $(G,\Delta)$ and it is that quantum group which acts on the deformed 
$C^*$-algebras.

However, already when $G$ is an ordinary  non-Abelian group,  constructing a nontrivial and  concrete 
dual unitary $2$-cocycle can be a very 
difficult task. For instance, in \cite{NT} the only  example given  is the one
canonically attached (see below) to the equivariant quantization map constructed in \cite{BG}. 
Moreover, even at the level of $C_0(G)$,  it is not clear whether the constructions of \cite{NT} and of
\cite{BG} agree, while it is known \cite{Nes-flat} to hold for of actions of $\R^d$. 
We should also mention that  the framework of \cite{Ri} and \cite{BG} comes
naturally with parameters and that Rieffel's  methods  are  perfectly 
well adapted to study the question of continuity
of the associated field of deformed $C^*$-algebras. In contrast, it is uncertain whether the methods
of \cite{NT} (and already  those of \cite{Nes-Shanga,K1}) applied in a parametric
situation can lead to results about continuity. Moreover, oppositely to Rieffel's type methods,
it is unclear whether those of \cite{NT} are well suited 
in view of applications in noncommutative geometry, for instance in spectral triple theory
\cite{Co4,Co5}.

 For all these reasons, and even if there exists now a satisfactory and completely general deformation 
 theory of $C^*$-algebras
 by use of its symmetries \cite{NT}, we believe that constructing deformation theories directly from 
equivariant quantizations is important in its own right.

The main goal of this paper is to continue the program initiated in \cite{BG} and which consists in
extending Rieffel's approach to deformation
for more general groups than $\R^d$. In \cite{BG}, even if we were in the relatively simple situation
 of solvable simply connected real Lie groups,
we  faced serious analytical difficulties underlying
the non-commutativity of the group. Here we study the somehow  opposite situation of Abelian but non-Lie
groups. More specifically, the groups we consider here are finite dimensional vector spaces over
a non-Archimedean local field
of characteristic different from $2$. 
At the root of our construction stands the $p$-adic version of the Weyl quantization introduced by Haran \cite{Haran} and further extended by
 Bechata \cite{Bechata} and Unterberger \cite{Unterberger}.
Even if our 
framework is already covered
by  Kasprzak's approach (in fact, it this one of our results), 
the primary interest of the present 
approach is to design new analytical tools adapted to the non-Lie case. In a forthcoming paper, we 
treat   a non-Abelian and non-Lie example, 
given by the affine group of a non-Archimedean local field.
Another important feature of the case studied here, is that the deformation parameter is no longer 
a real number.
Instead, our parameter space is the  ring of integers of the field. This affects substantially 
the answer we are able to give about the continuity of the field of deformed $C^*$-algebras. To conclude
with general features, we should also mention that here, part of the analytical arguments are even simpler
than their Archimedean analogues in \cite{Ri}. This a somehow recurrent phenomenon  in $p$-adic 
harmonic analysis. Here, this comes from the following reason.  The
$p$-adic pseudo-differential calculus is controlled by two operators $I$ and $J$ \cite{Bechata,Haran}, 
which are the natural non-Archimedean
substitutes for the operator of multiplication by the function $[x\in\R^n\mapsto(1+\langle x,x\rangle)^{1/2}]$ 
and for the flat Laplacian. But here they do commute! 
However, $p$-adic pseudo-differential operators do not commute in general!

\quad

Let us now be more precise about the program we wish to develop. In the differentiable setting, 
to define a non-formal equivariant quantization, one generally starts
 from a symplectic manifold $(M,\omega)$
together with a Lie subgroup $\tilde G$ of the group of symplectomorphisms. 
An equivariant quantization is a map 
$$
\bO:C^\infty_c(M)\to \CB(\CH_\pi),
$$
associated with a projective unitary  representation  $(\CH_\pi,\pi)$ of $\tilde G$,
satisfying the covariance property:
$$
\pi(g)\,\bO(f)\,\pi(g)^*=\bO(f^g)\,,\quad\forall f\in C^\infty_c(M)\,,\;\forall g\in\tilde G,
$$
where $f^g:=[x\in M\mapsto f(g^{-1}.x)]$. There is a paradigm of such equivariant 
quantizations, which covers most of the quantizations known, called ``Moyal-Stratonowich quantization'' 
by Cari\~nena, Gracia-Bond\'{\i}a and V\'arilly in  \cite{CGBV} (see also \cite[section 3.5]{GBFV}).
It is associated with a family of bounded (to simplify a little bit the picture)
selfadjoint operators $\{\Omega(x)\}_{x\in M}$ on $\CH_\pi$ satisfying the covariance property $\pi(g)\,
\Omega(x)\,\pi(g)^*
=\Omega(g.x)$ (plus two other properties that are not very relevant for the following discussion). 
The associated  quantization map is then defined by
$$
\bO(f):=\int_M\,f(x)\,\Omega(x)\, d\mu(x)\,,\quad\forall f\in C^\infty_c(M),
$$
where $d\mu$ is the Liouville measure on $M$.
Now, to connect equivariant quantization to deformation and to 
quantum groups, we need to restrict ourselves to the situation where $\tilde G$ possesses a subgroup $G$
which acts simply transitively on $M$. Under the identification $G\simeq M$, the Lie group
$G$ is then endowed
with a symplectic structure that is invariant under left  translations. Hence, what we are looking for 
is a non-formal quantization map on a the symplectic Lie group $G$ which is equivariant under left 
translations. In the context of Moyal-Stratonowich quantization, with
$e$  the neutral element of $G$, we then have:
 $$
 \Omega(g)=\pi(g)\,\Omega(e)\,\pi(g)^*,
 $$
 and  the Liouville measure $d\mu(x)$ on $M$ becomes a left 
invariant Haar measure $d^\lambda(g)$ on $G$. Hence, setting $\Sigma:=\Omega(e)$,
a $G$-equivariant Moyal-Stratonowich quantization on $G$ is always of the form
$$
\bO_{\pi,\Sigma}(f):=\int_G\,f(g)\;\pi(g)\,\Sigma\,\pi(g)^*\, d^\lambda(g)\,,\quad\forall f\in C^\infty_c(G).
$$
What is important with  the formula above is that symplectic differential geometry  disappeared
from the picture and provides an ansatz to construct  left-invariant quantizations on general groups.

Assume now that $G$ is an arbitrary locally compact second countable group, pick 
a projective representation $(\CH_\pi,\pi)$ and let $\Sigma\in\mathcal B(\CH_\pi)$. In general, there is no reason why
the associated quantization map behaves well.
A natural  assumption is that the quantization map  $\bO_{\pi,\Sigma}:C_c(G)
\to \CB(\CH_\pi)$ extends to a unitary operator from $L^2(G)$ to $\mathcal L^2(\CH_\pi)$, 
 the Hilbert space of Hilbert-Schmidt
operators on $\CH_\pi$. In this case we talk about unitary quantizations. In the existing examples, 
 the representation space $\CH_\pi$ is of the form $L^2(Q,\nu)$ and the basic operator $\Sigma$ is of
the form 
$m\circ T_\sigma$, where $m$ is an operator of multiplication by a Borelian function
on $Q$ and $T_\sigma$ is the operator of composition by a Borelian involution $\sigma:Q\to Q$.

For unitary quantizations, one can transfer the algebraic structure of $\mathcal L^2(\CH_\pi)$
to $L^2(G)$ and define an associative left equivariant deformed product:
$$
\star_{\pi,\Sigma}:L^2(G)\times L^2(G)\to L^2(G)\,,\quad (f_1,f_2)\mapsto \bO_{\pi,\Sigma}^*\big(
\bO_{\pi,\Sigma}(f_1)\,\bO_{\pi,\Sigma}(f_2)\big),
$$
where  $\bO_{\pi,\Sigma}^*:\mathcal L^2(\CH_\pi)\to L^2(G)$ denotes the adjoint map,
which is traditionally called the symbol map.
Note that  on the trace-class ideal $\mathcal L^1(\CH_\pi)\subset \mathcal L^2(\CH_\pi)$,
it is given by
$$
\bO_{\pi,\Sigma}^*(S)=\big[g\mapsto\Tr\big(S\,\pi(g)\,\Sigma\,\pi(g)^*\big)\big],\quad \forall S\in \mathcal L^1(\CH_\pi),
$$
so that the deformed product is then  associated with a distributional (in the sense of Bruhat
\cite{Bruhat})
tri-kernel:
$$
f_1\star_{\pi,\Sigma} f_2(g_0)=\int_{G\times G} K_{\pi,\Sigma}(g_0,g_1,g_2)\,f_1(g_1)\,f_2(g_2)\, d^\lambda(g_1)\,d^\lambda(g_2),
$$
where  $K_{\pi,\Sigma}$ is (formally) given  by 
$$
K_{\pi,\Sigma}(g_0,g_1,g_2)=\Tr\big(
\Sigma\,\pi(g_0^{-1}g_1)\,\Sigma\,\pi(g_1^{-1}g_2)\,\Sigma\,\pi(g_2^{-1}g_0)\big).
$$
In general, $K_{\pi,\Sigma}$ is not a singular object but rather 
a regular function  (in the sense of Bruhat \cite{Bruhat}).

There is then a natural candidate for a dual unitary $2$-cocycle $F_{\pi,\Sigma}$ on $G$, namely
$$
F_{\pi,\Sigma}:=\int_{G\times G} \overline{K_{\pi,\Sigma}(e,g_1,g_2)}\,\lambda_{g_1^{-1}}\otimes\lambda_{g_2^{-1}}\,
 d^\lambda(g_1)\,d^\lambda(g_2)
\in W^*(G\times G).
$$
The $2$-cocyclicity  property is automatic from the construction since this property is equivalent to left-equivariance
and associativity of the deformed product $\star_{\pi,\Sigma}$. The only  remaining task is to check 
that $F_{\pi,\Sigma}$ is well defined as a unitary element of the  group von Neumann algebra 
$W^*(G\times G)$. As observed in \cite{NT}, this is the case for the quantization map
considered in \cite{BG}.

There is also a natural candidate for a deformation theory. Consider  now
a $C^*$-algebra $A$ endowed with a continuous action $\alpha$ of $G$. 
Then, we may try to define a new multiplication on $A$ by the formula:
\begin{align}
\label{base}
a\star^\alpha_{\pi,\Sigma} b:=\int_{G\times G} K_{\pi,\Sigma}(e,g_1,g_2)\,\alpha_{g_1}(a)\,\alpha_{g_2}(b)\,
 d^\lambda(g_1)\,d^\lambda(g_2)\,,\quad a,b\in A.
\end{align}
Of course, there no reason why this integral should be well defined since the map $[g\mapsto\alpha_g(a)
]$ is constant in norm and since $K_{\pi,\Sigma}$ 
is typically unbounded (at least when the group is non-unimodular).
Rieffel's approach to deformations consists then in two steps:
\begin{itemize}
\item Find $A_{\rm reg}$, a dense  $\alpha$-stable Fr\'echet subalgebra of $A$, on which the
multiplication  \eqref{base} is inner.
\item Embed continuously the deformed Fr\'echet algebra $(A_{\rm reg},\star^\alpha_{\pi,\Sigma})$ 
into a $C^*$-algebra.
\end{itemize}
The $C^*$-deformation of $A$ is then defined as the $C^*$-completion of $A_{\rm reg}$ and is denoted
$A_{\pi,\Sigma}$.

To deal with the first step, one usually works with oscillatory integrals. Roughly speaking, it boils down to
find a countable family of operators $\underline\bD:=\{\bD_j\}_{j\in J}$, where $J$ is the index set
of  the seminorms of $A_{\rm reg}$, acting on the space of regular functions (in the sense of Bruhat) $\CE(G\times G)$, which leave invariant the two-point kernel 
$\bD_jK_{\pi,\Sigma}(e,.,.)=K_{\pi,\Sigma}(e,.,.)$ and 
such that for all $a,b\in A_{\rm reg}$, the transposed operator $\bD_j^t$ sends  the mapping
$[(g_1,g_2)\mapsto \alpha_{g_1}(a)\,\alpha_{g_2}(b)]$
to an element of $L^1(G\times G,A_{{\rm reg},j})$,  where $A_{{\rm reg},j}$ denotes the semi-normed space associated with the $j$-th seminorm of $A_{{\rm reg}}$.
One then gets a continuous bilinear map defined by
$$
\star^\alpha_{\pi,\Sigma}:A_{\rm reg}\times A_{\rm reg}\to A_{{\rm reg},j}\,,\quad (a,b)\mapsto
\int_{G\times G} K_{\pi,\Sigma}(e,g_1,g_2)\,\bD_j^t\big(\alpha_{g_1}(a)\,\alpha_{g_2}(b)\big)\,
 d^\lambda(g_1)\,d^\lambda(g_2).
$$
Then, it remains to show that the associativity is preserved by the regularization scheme underlying 
the introduction of the operators $\underline\bD$.

For the second point, one usually starts  by proving an $A$-valued version of the Calderon-Vaillancourt
Theorem. It  basically says that if you consider $A$ to be the $C^*$-algebra of right uniformly continuous
bounded functions on $G$ with action given by right translation, then the quantization map 
$\bO_{\pi,\Sigma}$ should 
send continuously $A_{\rm reg}$ to $\CB(\CH)$. When the projective representation $\pi$  is square 
integrable\footnote{By a square integrable projective representation, we mean a representation of the 
associated central extension which is square integrable
modulo its center.},  from the Duflo-Moore theory \cite{Duflo-Moore} we can construct a 
weak resolution of identity from $\pi$. It allows to use general methods based on Wigner functions as first
introduced in \cite{Unold}.

\quad

The paper is organized as follow. In section 2, we fix notations 
 and we review the $p$-adic Weyl pseudo-differential calculus 
 on $\k^{d}$, where $\k$ is a non-Archimedean
 local field of characteristic different from $2$. 
 In fact, we consider a family of $p$-adic quantization maps, indexed by a parameter
 $\theta$ in $\CO_\k$, the ring of integers of $\k$.
Section 3 contains the most technical part of the paper. It is in that section that we construct the
space $A_{\rm reg}$ of regular elements  of a $C^*$-algebra $A$ for a given continuous action
$\alpha$ of $\k^{2d}$ (Definition \ref{Areg}).
We then define a deformation theory at the Fr\'echet level (Theorem \ref{TDF}), using  oscillatory
integrals methods.
In section 4, we extend the $p$-adic Calderon-Vaillancourt Theorem of \cite{Bechata} in the case
of $C^*$-valued symbols (Proposition \ref{CV}). 
This yields an embedding of the deformed Fr\'echet algebra into 
a $C^*$-algebra and consequently a deformation theory at the $C^*$-level (Theorem \ref{HW}).
We call $A_\theta$ the $C^*$-deformation of $A$.
We also prove that our deformed $C^*$-norm can be realized as the $C^*$-norm of $A$-linear
adjointable bounded endomorphisms of a $C^*$-module (Proposition \ref{CMA}) and
that our construction coincides with those of \cite{K1} and \cite{Nes-Shanga} (Theorem \ref{TT}).
In the final section 5, we establish the basic properties of the deformation. In particular, 
we show that contrary to the Archimedean case,
the $K$-theory is not an invariant of the deformation and that the fields of deformed
$C^*$-algebras $(A_{\gamma\theta^2})_{\theta\in\CO_\k}$, for $\gamma\in\CO_\k$ arbitrary, are
 continuous.

\section{A $p$-adic pseudo-differential calculus}
\subsection{Framework and notations}

 Let   $\k$ be a non-Archimedean local field, that is a non-Archimedean  non-discrete 
 locally compact topological field.  $\k$ is  complete for the ultrametric associated with
 the  absolute value $|.|_\k$, given  by the restriction to dilations of the module function  (and extended to 
zero on $0$). 
 Non-Archimedean local fields are classified.
In characteristic zero,  $\k$ is  isomorphic to a finite extension  of $\Q_p$, the field of $p$-adic numbers.
In positive characteristic, $\k$ is isomorphic to  $\mathbb{F}_{q}((X))$, the field of Laurent series with 
coefficients 
 in a finite field. For important technical reasons,
 we will  (mostly) assume that the
 characteristic of $\k$ different  from $2$. The additive group
$(\k,+)$ is    self-dual,  with isomorphism $\k\simeq
 {\rm Hom}\big(\k,\mathbb U(1)\big)$ given by
 $x\mapsto \Psi(x.)$,
where $\Psi$ is a fixed non-trivial  character.
We denote by $\CO_{\k}:=\{x\in\k:|x|_\k\leq 1\}$ the ring of integers,
by $\varpi$ the generator its unique maximal ideal  ($\varpi$ is called the uniformizer and satisfies
$|\varpi|_\k=({\rm Card}(\CO_\k/\varpi\CO_\k))^{-1}$)
and by  $\CO_\k^o$ 
the conductor of $\Psi$, that is to say the largest ideal\footnote{$\CO_\k^o$ is of the form $\varpi^{n(\Psi)}\CO_\k$, with n($\Psi$) $\in \Z$ uniquely defined by the character $\Psi$.}  of $\CO_\k$ on which $\Psi$ is constant.
We normalize the Haar measure of $(\k,+)$ (denoted by $dx$) by requiring  it  to be selfdual with respect to the duality associated with $\Psi$ or, equivalently, by requiring that ${\rm Vol}(\CO_\k)\times
{\rm Vol}(\CO_\k^o)=1$.

For example, if $\k=\Q_p$, we have $\CO_\k=\Z_p$, the ring of $p$-adic 
integers, $\varpi=p$ and $|x|_{\Q_p}=p^{-k}$ if 
$x = p^{k}\frac{m}{n}\in \Q$ (where $m$ and $n$ are integers 
non-divisible by $p$).
If one  chooses the (standard) character:
$$
 \Psi_0(x):= \exp\Big\{2 \pi i\sum_{\mathclap{-n_0 \leq n <0}}a_n p^n
\Big\}\quad\mbox{if}\quad x=\sum_{n \geq -n_0} a_n p^n,
$$
we find $\CO_\k^o = \Z_p$ and our normalization for the Haar measure reads ${\rm Vol}(\Z_p)=1$.

We let $|.|^{\vee}_\k$ be the second ultra-metric norm  on $\k$,
given by $|x|_\k^{\vee}=|x\varpi^{-n(\Psi)}|_\k$.  More generally, we let 
$|.|_{\k^d}$,  $|.|_{\k^d}^{\vee}$ be the associated  $\sup$-norms on $\k^d$, $d\in\N$:
$$
|x|_{\k^d}=\max_{1 \leq i \leq d} |x_i|_\k\quad \mbox{and}\quad |\xi|_{\k^d}^{\vee}=\max_{1 \leq i \leq d}
|\xi_i|_\k^{\vee}.
$$
For $x,y,\xi,\eta\in\k^d$ we set $X=(x,\xi),Y=(y,\eta)\in\k^{2d}$ and we consider the symplectic structure:
\begin{align}
\label{sp}
[.,.]:\k^{2d}\times\k^{2d}\to\k\,,\quad (X,Y)\mapsto \langle y,\xi\rangle - \langle x,\eta\rangle ,
\end{align}
where  $ \langle x,y\rangle=\sum_{j=1}^d x_iy_j$.

The following numerical function plays a decisive role in our analysis:
\begin{align}
\label{mu}
\mu_0(X) :=\max\{1,|2x|_{\k^d},|2\xi|_{\k^d}^{\vee}\},\quad X=(x,\xi)\in\k^{2d}.
\end{align}
From the ultrametric inequality, one sees that $\mu_0$  is invariant under translations in 
$(\tfrac12\CO_\k)^d\times (\tfrac12\CO_\k^o)^d$. It is known (see \cite{Haran}) that  
$\mu_0^{-1}$  belongs to 
$L^p(\k^{2d})$ for all $p>2d$ 
and  satisfies a Peetre type inequality:
\begin{align}
\label{Peetre}
 \mu_0(X+Y)\leq\mu_0(X)\,\mu_0(Y),\quad\forall X,Y\in\k^{2d}.
 \end{align}

Let  $\CE(\k^{2d})$ (resp. $\CD(\k^{2d})$, $\CD'(\k^{2d})$) be the set of smooth 
functions (resp. smooth and compactly supported functions, distributions)
 in the sense of Bruhat \cite{Bruhat}. Since $\k^{2d}$ is totally disconnected,
$\CE(\k^{2d})$ consists in locally constant functions and $\CD(\k^{2d})$ 
consists in  locally constant compactly supported functions. 
In particular, $\mu_0$ belongs to $\CE(\k^{2d})\subset \CD'(\k^{2d})$. 
 $\CD(\k^{2d})$ and  $\CD'(\k^{2d})$ are
  stabilized by the (selfdual) Fourier transform and by its symplectic variant:
\begin{align}
\label{SFT}
\big(\mathcal{G}f\big)(X) =|2|_\k^{d}  \int_{\k^{2d}} f(Y) \Psi(2[Y,X]) dY.
\end{align}
The symplectic Fourier transform 
extends  to a unitary operator on $L^2(\k^{2d})$ and is its own inverse.

\begin{rmk}
The normalization chosen in the definition of $\mu_0$ and $\CG$ (i.e$.$ the factor 
$2$) allows to simplify some computations but is by no mean the reason why we have to exclude the 
characteristic 2.
\end{rmk}

From our perspectives, the Schwartz-Bruhat space\footnote{Indeed, since $\k^{2d}$ is totally 
disconnected,
$\CD(\k^{2d})$ coincides with the Schwartz space, as defined in 
\cite[num\'ero 9]{Bruhat}.}  $\CD(\k^{2d})$ is not  suitable. 
Instead, we consider a variant of it, that we may define with the help of two 
unbounded operators on $L^2(\k^{2d})$.
Let  $I$ be  the operator of point-wise multiplication by the function $\mu_0$:
\begin{align}
I\vf(X):=\mu_0(X)\,\vf(X),
\end{align}
 and   $J$ the  convolution operator by the Bruhat distribution $\CG(\mu_0)$:
\begin{align}
J:=\CG\circ I\circ\CG.
\end{align}
More generally, we denote by $J^s$,  $s\in\R$, the  convolution operator by  
$\CG(\mu_0^s)$.
The Bruhat  distributions $\CG(\mu_0^s)$ are known to be 
supported in $( \tfrac12\CO_\k)^d\times 
(\frac12\CO_\k^o)^d$ (we will give an elementary proof of this fact in Lemma \ref{lem1}).
The operator $J$ has to be  considered as a substitute  of an order one elliptic differential operator,
in the dual sense
that $\mu_0$ has to be considered as a substitute for a radial coordinate function on $\k^{2d}$. Since 
$\mu_0$ is $( \tfrac12\CO_\k)^d\times 
(\frac12\CO_\k^o)^d$-locally constant and
$\CG(\mu_0)$ is supported on $( \tfrac12\CO_\k)^d\times 
(\frac12\CO_\k^o)^d$, as continuous operators on $\CD'(\k^{2d})$, $I$ and $J$ commute! 
As unbounded operators on  $L^2(\k^{2d})$ (with  initial domain $\CD(\k^{2d})$)
 they  are  essentially selfadjoint and positive. Following Haran \cite{Haran}, we introduce
the another analogue of the Schwartz space:
$$
\CS(\k^{2d}) := \bigcap_{n,m\in\N} {\rm Dom}(I^nJ^m) 
:= \big\{\vf \in L^2(\k^{2d}) \,:\, \forall n,m\in\N, I^n J^m\vf \in L^2(\k^{2d})\big\}.
$$
Equipped with the seminorms:
\begin{align}
\label{SN0}
\vf\mapsto\|I^nJ^m\vf\|_2, \quad n,m\in\N,
\end{align}
 $\CS(\k^{2d})$ becomes   Fr\'echet and nuclear
and of course
$$
\CD(\k^{2d}) \subset \CS(\k^{2d})\subset C(\k^{2d}),
$$
continuously. Moreover, $\CD(\k^{2d})$ is dense in $\CS(\k^{2d})$ but the inclusion is proper
since an element in $\CS(\k^{2d})$ does not need to be locally constant nor compactly 
supported.
Since $\mu_0^{s}$ belongs to $L^1(\k^{2d})$ for $s<-2d$, 
in the seminorms \eqref{SN0}, we can change the $L^2$-norm with any other $L^p$-norm, 
$p\in[1,\infty]$, while keeping the same topology\footnote{In section \ref{FDCA} we will
 consider the seminorms \eqref{SN0} with $p=\infty$ instead of $2$, see \eqref{SN}.}.
We let $\CS'(\k^{2d})$ be the strong dual 
of $\CS(\k^{2d})$, that we call the space of tempered 
distributions. The operators $I$ and $J$ extend to
continuous endomorphisms of $\CS'(\k^{2d})$ and, of course, still commute there.
We can also define the $d$-dimensional versions of the 
Schwartz space $\CS(\k^d)$ and of its dual  $\CS'(\k^d)$ by
considering the $d$-dimensional version of the function $\mu_0$ and the $d$-dimensional 
ordinary Fourier transform  to define 
$d$-dimensional version of the operators $I$ and $J$.

The following (almost obvious) properties will be used repeatedly:
\begin{lem} 
\label{lem1} 
(i) For $Y\in\k^{2d}$, set $\mu_Y(X):=\mu_0(X-Y)$. Then for all $s,t_1,\dots,t_n
\in\R$, $Y_1,\dots,Y_n\in\k^{2d}$,   we have:
$$
I^s\,\mathcal G(\mu_{Y_1}^{t_1}\dots\mu_{Y_n}^{t_n})
=\mathcal G(\mu_{Y_1}^{t_1}\dots\mu_{Y_n}^{t_n})
\quad
\mbox{and} 
\quad
J^s(\mu_{Y_1}^{t_1}\dots\mu_{Y_n}^{t_n})
=\mu_{Y_1}^{t_1}\dots\mu_{Y_n}^{t_n},
$$
(ii) For $Y\in\k^{2d}$, set $\Psi_Y(X) := \Psi(2[X,Y]) $. Then for all $s\in\R$,
we have $J^s \Psi_Y=\mu_0^s(Y)\Psi_Y$.
\end{lem}
\begin{proof}

 (i) Of course, the two equalities we have to prove are equivalent.
 Set $O:= (\tfrac12\CO_\k)^d\times(\tfrac12\CO_\k^o)^d$.
  Fix $t>2d$ and observe that 
  $\mu_0^{-t}\in L^1(\k^{2d})$ so that $\CG(\mu_0^{-t})\in C_0(\k^{2d})$.
 Since moreover $\mu_0^{-t}$ is $O$-invariant, we have
 $$
\CG(\mu_0^{-t})(X)= \overline{\Psi}(2[X,Y])\,\CG(\mu_0^{-t})(X)\,,\qquad
\forall (Y,X)\in O\times \k^{2d}.
$$
This implies that  $\CG(\mu_0^{-t})$ is supported in $O$
since one can easily  construct a pair  $(Y,X)\in O\times \k^{2d}\setminus O$ such that $\Psi(2[X,Y])\ne 1$.
As $\mu_0$ equals one on $O$, we get for all $s>0$, 
$I^s \CG(\mu_0^{-t})=\CG(\mu_0^{-t})$ which is equivalent to 
$J^s(\mu_0^{-t})=\mu_0^{-t}$. Using that $J$ commutes with $I$ and 
 translations  $\tau_Yf(X)=f(X+Y)$, the result follows by applying $J^s$ to
 the identity
  $$
\mu_{Y_1}^{t_1}\dots\mu_{Y_n}^{t_n}= \tau_{Y_1}I^{t_1}
\tau_{Y_2-Y_1}I^{t_2}\dots\tau_{Y_n-Y_{n-1}}I^{t_n+t} (\mu_0^{-t}).
$$

(ii) We denote by $\langle.,.\rangle$ the (bilinear) duality pairing between $\CS(\k^{2d})$
and $\CS'(\k^{2d})$.  Fix $X\in\k^{2d}$. 
Since $\Psi_X\in C_b(\k^{2d})$, we may view it as an element of  $\CS'(\k^{2d})$.
Then, we have
for all $\vf\in \CS(\k^{2d})$:
$$
\CG\vf(X)=|2|_\k^d\langle\Psi_X,\vf\rangle
\quad\mbox{and thus}\quad\vf(X)=|2|_\k^d\langle\Psi_X,\CG\vf\rangle=|2|_\k^d\langle\CG\Psi_X,\vf\rangle.
$$
From this and the fact that $J^s=(\CG I \CG)^s=\CG I^s\CG$, we get
\begin{align*}
\langle J^s\Psi_X,\vf\rangle&=\langle\CG I^s\CG\Psi_X,\vf\rangle
=\langle\CG\Psi_X, I^s\CG\vf\rangle=|2|_\k^{-d}I^s\CG\vf(X)=|2|_\k^{-d}
\mu_0^s(X)\CG\vf(X)=\mu_0^s(X)\langle\Psi_X,\vf\rangle.
\end{align*}
This completes the proof.
\end{proof}

\subsection{Weyl quantization  on local fields}

In this subsection we recall some facts about the $p$-adic pseudo-differen\-tial calculus introduced
by Haran in \cite{Haran} and further studied by Bechata and Unterberger \cite{Bechata,Unterberger}
(see also \cite{Segal,Weil} for a completely general construction of the Weyl quantization).
 We assume that the characteristic of $\k$ is different from $2$. 
We  fix $\theta\in\k^\times$. It will play the role of the deformation parameter.
For any tempered distribution $F \in \CS'(\k^{2d})$, we denote by ${ \bO_\theta}(F)$ 
the continuous linear operator from $\CS(\k^d)$ to $\CS'(\k^d)$ defined (with a little
abuse of notation) by:
\begin{align}
\label{QM}
\bO_\theta(F):\CS(\k^d)&\to \CS'(\k^d)\nonumber\\
\vf&\mapsto  \Bigg[\phi\in\CS(\k^d)\mapsto
|\theta|^{-d}_\k \int_{\k^d}\Bigg(\int_{\k^{2d} } F\big(\tfrac12(x+y),\eta\big)\,\vf(y)\Psi(\theta^{-1}\langle x-y,\eta \rangle) \,d\eta \,dy\,\Bigg)\phi(x)
 \,dx\Bigg].
 \end{align}
 The  
 distribution $F$ is called the symbol of the pseudo-differential operator $\bO_\theta(F)$.
This Weyl type pseudo-differential calculus
 is covariant under the action of the additive group
 $\k^{2d}$ by translations, in the sense that
 \begin{align}
 \label{cov}
 U_\theta(Y)\,\bO_\theta(F)\,U_\theta(Y)^*=\bO_\theta(\tau_{-Y} F),\qquad\forall F\in \CS'(\k^{2d}),\quad
 \forall X\in\k^{2d},
 \end{align}
 where $\tau_YF(X):=F(X+Y)$ and where $U_\theta$ is the projective unitary 
 (Schr\"odinger) representation of $\k^{2d}$ on $L^2(\k^d)$ given, for
 $X=(x,\xi)\in\k^d\times\k^d$, by
 \begin{align}
 \label{SCH}
 U_\theta(X)\vf(y):=\Psi\big(\theta^{-1}\langle \xi,y-\tfrac12 x\rangle\big)\,\vf(y-x).
 \end{align}
 The covariance property is obvious for a symbol  $F\in L^1(\k^{2d})$, as seen from the 
absolutely convergent integral representation:
 \begin{align}
 \label{SF}
 \bO_\theta(F)=\big|\tfrac2\theta\big|^{d}_\k\int_{\k^{2d}} F(X)\,\Omega_\theta(X)\,dX,
 \end{align}
 where 
 $$
 \Omega_\theta(X):=U_\theta(X)\,\Sigma \,U_\theta(X)^*\,,\quad\forall X\in \k^{2d},
 $$
 and where
 $\Sigma$ is the selfadjoint  involution on $L^2(\k^d)$ given by $\Sigma\vf(x)=\vf(-x)$.
 Note the scaling relations
 $$
 U_\theta(x,\xi)=U_1(x,\theta^{-1}\xi),\qquad \Omega_\theta(x,\xi)=\Omega_1(x,\theta^{-1}\xi)\,,\qquad
 \forall x,\xi\in\k^d.
 $$
The integral representation given above implies that when $F\in L^1(\k^{2d})$, the pseudo-differential
operator $\bO_\theta(F)$ is  bounded, with\footnote{In characteristic different from $2$, $|2|_\k=1$ or
$|2|_\k=\tfrac12$.}:
\begin{align}
\label{BB}
\|\bO_\theta(F)\|\leq |\theta|^{-d}_\k\,\|F\|_1.
\end{align}
Of course, this inequality blows up in the limit $\theta\to0$.
By Fourier theory (on selfdual locally compact Abelian groups), and when
characteristic of $\k$ is different from $2$, one sees that the
associated quantization  map 
 $$
 \bO_\theta:\CS'(\k^{2d})\to \mathcal L\big(\CS(\k^{d}),\CS'(\k^{d})\big),\quad F\mapsto\bO_\theta(F),
 $$
 restricts to $|\theta|^{-d/2}_\k$  times a unitary operator from $L^2(\k^{2d})$
 to the Hilbert space of Hilbert-Schmidt operators on $L^2(\k^d)$. Hence, we also have the bound:
 \begin{align}
\label{BBB}
\|\bO_\theta(F)\|\leq\|\bO_\theta(F)\|_2= |\theta|^{-d/2}_\k\,\|F\|_2.
\end{align}
 One can then transport
 the algebraic structure of the Hilbert-Schmidt operators to $L^2(\k^{2d})$, by setting
 $$
 f_1\star_\theta f_2:=\bO_\theta^{-1}\big(\bO_\theta(f_1)\,\bO_\theta(f_2)\big)\,,\quad\forall f_1,f_2\in
 L^2(\k^{2d}).
 $$
 At the level of the Schwartz space, this deformed product has a
 familiar form:
 \begin{equation}
 \label{M}
 f_1\star_\theta f_2(X)=|\tfrac2\theta|^{2d}_\k\int_{\k^{2d}\times\k^{2d}} \overline\Psi\big(\tfrac2\theta
 [Y-X,Z-X]\big)\,f_1(Y)\,f_2(Z)\,dY\,dZ,
 \qquad\forall f_1,f_2\in\CS(\k^{2d}).
 \end{equation}
 Indeed, this is the $p$-adic version of the Moyal product in its integral form.
 Note that this relation can be rewritten as a functional identity:
 $$
  f_1\star_\theta f_2=|2|^{2d}_\k\int_{\k^{2d}\times\k^{2d}} \overline\Psi\big(2
 [Y,Z]\big)\,\tau_{\theta Y}(f_1)\,\tau_Z(f_2)\,dY\,dZ,
 \qquad\forall f_1,f_2\in\CS(\k^{2d}),    
 $$
 where $\tau$ is the translation operator $\tau_{\theta Y}(f_1)(X)=f_1(X+Y)$.
 The important observation is that this formula 
  makes sense even when $\theta=0$. Indeed,  in this case  we have
$ f_1\star_{\theta=0} f_2=f_1\,\CG^2( f_2)=
f_1\,f_2$.
\begin{rmk}
In  characteristic $2$, one can formally change the character $\Psi$ to $\Psi(\tfrac12.)$
in \eqref{M} while preserving  its fundamental properties of associativity and covariance.
The corresponding  modification in  \eqref{QM} is
to  suppress  the ill-defined  factor $\frac12$. But then, the operator kernel of $\bO_\theta(F)$
will  up to a constant\footnote{$\CF_2$ is the partial Fourier transform on the second set of variables.}  
be $(\CF_2F)(x+y,\theta^{-1}(x-y))$.
Since the matrix $\begin{psmallmatrix}1&1\\\theta&-\theta\end{psmallmatrix}$ is not invertible in 
characteristic $2$, we loose the crucial  
property of unitarity (from $L^2$-symbols to Hilbert-Schmidt operators) of the quantization map.
\end{rmk}

  Set further
 $$
 \CB(\k^{2d}):=\big\{F\in \CS'(\k^{2d})\,:\, J^nF\in L^\infty(\k^{2d}),\;\forall n\in\N\big\}.
 $$
 Using coherent states and Wigner functions methods, 
 Bechata proved in \cite{Bechata} an analogue of the Calderon-Vaillancourt Theorem
 for the space $\CB(\k^{2d})$. Namely, he proved the following estimate:
 \begin{align}
 \label{CVB}
 \|\bO_\theta(F)\|\leq \|\mu_0^{-2d-1}\|_1\,\|J^{2d+1}F\|_\infty,\quad F\in  \CB(\k^{2d}),
 \end{align}
 where the norm on the left hand side denotes the operator norm on $L^2(\k^d)$.
 Contrarily to  \eqref{BB} and \eqref{BBB},
this inequality does not blow up in the limit $\theta\to 0$. The methods 
 leading to this key result rely on a clever redefinition of $\bO_\theta(F)$  in term of a quadratic
 form constructed out of specific  coherent states and Wigner functions. 
 Since we will borrow part of Bechata's techniques, we recall  some ingredients  of his
 construction.
  
  For $\vf\in L^{2}(\k^d)$, $\theta\in\k^\times$ 
   and $X\in\k^{2d}$, set $\vf^{\theta}_X:= U_\theta(X)\vf$,  where $U_\theta$ 
  is the projective representation of 
  $\k^{2d}$   given in \eqref{SCH}.
  It is known that  $U_\theta$ is square integrable modulo its center and that the 
  following
  reproducing formula holds:
  \begin{align}
  \label{RI}
  \langle\phi,\psi\rangle=|\theta|^{-d}_\k\|\vf\|^{-2}\int_{\k^{2d}}\langle\phi,\vf_X\rangle\langle\vf_X,\psi\rangle\,dX
  \,,
  \quad \forall \phi,\psi,\vf\in L^2(\k^d),\; \vf\ne 0.
  \end{align}
  Let then
  \begin{align}
  \label{Wigner}
  W_{\phi,\psi}^\theta(X):=\langle\phi,\Omega_\theta(X)\psi\rangle,\quad X\in\k^{2d},
  \end{align}
  be the Wigner function\footnote{The Wigner function 
  $ W_{\phi,\psi}^\theta$  is the symbol of the rank-one 
  operator $\vf\mapsto\langle\psi,\vf\rangle\phi$.} associated with the pair of vectors
  $\phi,\psi\in L^2(\k^d)$. 
Let  $\eta$ be the characteristic function of $\CO_\k^d$, normalized by $\|\eta\|_2 = 1$.
By \cite[eq. (1.9)]{Bechata}, we have for $X=(x,\xi)\in\k^{2d}$
and $\alpha,\beta\in\R$:
$$
\tilde I^{\alpha}\tilde J^{\beta}\eta^\theta_X = \mu_0^\alpha(x,0)\, \mu_0^\beta(0,\theta^{-1}\xi)\, \eta^\theta_X,
$$
where $\tilde I$, $\tilde J$ denote the $d$-dimensional versions of the operators $I$ and $J$.
In particular, the relation above entails that $\eta^\theta_X\in\CS(\k^d)$.
For $X,Y\in\k^{2d}$, we set $W_{X,Y}^\theta$ 
for the Wigner function associated with the pair of coherent  states 
$\eta^\theta_X,\eta^\theta_Y$:
\begin{align}
\label{Wigner2}
W_{X,Y}^\theta(Z):=W^\theta_{\eta^\theta_X,\eta^\theta_Y}(Z)=\langle\eta^\theta_X,\Omega_\theta(Z)\eta^\theta_Y\rangle
=\langle U_\theta(X)\eta,\Omega_\theta(Z)U_\theta(Y)\eta\rangle.
\end{align}

The next statement is extracted from \cite[Proposition 2.10 and Lemme 3.1]{Bechata}.
\begin{lem}
\label{Bech}
Let $\theta\in\k^\times$. For $X=(x,\xi)\in\k^{2d}$, set $X_\theta:=(x,\theta^{-1}\xi)\in\k^{2d}$.
Then for all $X,Y,Z\in\k^{2d}$ and with $\Phi$ the characteristic function of  $(\frac12\CO_\k)^d\times 
(\frac12\CO_\k^o)^d$, we have:
$$
\big|W^\theta_{X,Y}(Z)\big|=|2|_\k^d\,\Phi\big(Z_\theta-\tfrac12(X_\theta+Y_\theta)\big).
$$
Moreover,  $W^\theta_{X,Y}\in\CS(\k^{2d})$ and for all $n,m\in\Z$ we have:
$$
I^m J^n W^\theta_{X,Y}= \mu_0^m\big(\tfrac12(X+Y)\big)\,\mu_0^n\big(\tfrac1{2\theta}(X -
Y )\big)\,W^\theta_{X,Y}.
$$
\end{lem}

\section{The Fr\'echet deformation of a $C^*$-algebra}
\label{FDCA}
In this section, we fix  a $C^*$-algebra $A$, together
with a  continuous   action $\alpha$ of the
additive group $\k^{2d}$.   This yields a map
\begin{align}
\label{alphamap}
\tilde\alpha: A\to C_{b}(\k^{2d}, A), \qquad a\mapsto[X\mapsto\alpha_X(a)].
\end{align}
Fixing a faithful representation $\pi$ of  $A$ on $\CB(\CH)$,  we will frequently identify $A$ 
 with its image on $\CB(\CH)$.

Our first goal is to find $A_{\rm reg}$, a dense and $\alpha$-stable
Fr\'echet subalgebra of $ A$, on which we can give 
a meaning to the natural generalization of the deformed product  \eqref{M}:
\begin{equation}
\label{SP}
 a\star_\theta^\alpha b:=|2|_\k^{2d}
 \int_{\k^{2d}\times\k^{2d}} \overline\Psi\big(2[Y,Z]\big)\,\alpha_{\theta Y}(a)\,\alpha_Z(b)\,dY\,dZ,
 \qquad\forall a,b\in A_{\rm reg}.
 \end{equation}
Having in mind Bechata's version of the Calderon-Vaillancourt estimate \eqref{CVB}, there 
is an obvious candidate for $ A_{\rm reg}$, namely the set of elements   $a$ in $A$ which are such that
 $\tilde\alpha(a)\in\CB(\k^{2d},A)$ (see Definition \ref{B}). 

\subsection{Spaces of $ A$-valued functions and distributions}

Set  $C_b(\k^{2d}, A)$  
for the $C^*$-algebra  of $ A$-valued  continuous and bounded functions on $\k^{2d}$, with norm:
$$
\mathfrak{P}^A_{0}(F):=\sup_{X\in\k^{2d}}\|F(X)\|_A,
$$
and let $C_{u}(\k^{2d}, A)$ be the $C^*$-algebra of $ A$-valued uniformly continuous and bounded 
functions on $\k^{2d}$. The latter space is the maximal sub-$C^*$-algebra of $C_b(\k^{2d}, A)$ on which
the action $\tau\otimes\Id$ of $\k^{2d}$  is   continuous. 
Set then $\CS(\k^{2d}, A):= \CS(\k^{2d})\widehat\otimes A$ 
for the $ A$-valued version of the 
Schwartz space (recall that $\CS(\k^{2d})$ is  nuclear). We naturally embed $\CS(\k^{2d}, A)$
into $C_{u}(\k^{2d}, A)$. Since the (pair-wise commuting)  linear maps
 $I^nJ^m$, $n,m\in\Z$, are continuous on $\CS(\k^{2d})$, $I^nJ^m\otimes \Id$ (originally defined 
on the algebraic tensor product $ \CS(\k^{2d}) \otimes A$) extends to a continuous linear map
on  $\CS(\k^{2d}, A)$. To lighten our notations, and when no confusion can occur,  we will denote their 
extensions 
 by $I^nJ^m$. In a similar way, we will use the symbol $\CG$ to denote the continuous extension
of $\CG\otimes\Id$ on $\CS(\k^{2d}, A)$.
As already mentioned,    the Fr\'echet topology 
of $\CS(\k^{2d}, A)$ can be alternatively described via the seminorms:
\begin{align}
\label{SN}
\mathfrak{P}^A_{n,m}(f):=\sup_{X\in\k^{2d}}\big\|\big(I^nJ^m  f\big)(X)\big\|_A\,,\quad n,m\in\N.
\end{align}
Since $\CS(\k^{2d})$ is Fr\'echet and nuclear, its strong dual $\CS'(\k^{2d})$ 
 is also nuclear (see for instance \cite[Proposition 50.6]{Treves}).
 Therefore, we shall denote by 
$\CS'(\k^{2d}, A)$ the completed tensor product 
$\CS'(\k^{2d})\widehat\otimes A$. Note that
by \cite[P. 525]{Treves},
$\CS'(\k^{2d}, A)$  identifies isometrically with 
the space of continuous linear mappings from $\CS(\k^{2d})$
to $ A$. Under this identification, we get an  embedding of $C_b(\k^{2d}, A)$ into $\CS'(\k^{2d}, A)$.
Since the operators $I^nJ^m$, $n,m\in\Z$, act continuously (by transposition) on $\CS'(\k^{2d})$,
they extend  continuously on $\CS'(\k^{2d}, A)$ and we still denote them by
$I^nJ^m$. 
Similarly, we denote by $\mathcal G$ the continuous extension of the symplectic Fourier transform on 
 $\CS'(\k^{2d}, A)$.
The next space we introduce is of our principal tools:

\begin{dfn}
\label{B}
For $ A$  a $C^*$-algebra, we set
$$
 \CB(\k^{2d}, A):=\big\{F\in \CS'(\k^{2d}, A)\,:\,\forall n\in\N, \,J^n F\in L^{\infty}(\k^{2d}, A)
\big\}.
 $$
\end{dfn}

We  endow $\CB(\k^{2d}, A)$ with the topology associated with the following family of
seminorms:
\begin{align}
\label{SNBN}
\mathfrak{P}^A_{n}(F):=\sup_{X\in\k^{2d}}\big\|\big(J^n F\big)(X)\big\|_A \,,\quad\forall n\in\N.
\end{align}
When  $ A=\C$, we 
denote these seminorms 
by $\mathfrak{P}_{n}$.  Identifying in a natural way the algebraic tensor product $\CB(\k^{2d})\otimes
A$ with a subspace of $\CB(\k^{2d},A)$,
 it is easy to see that $\mathfrak{P}^A_{n}$ is a 
cross-seminorm:
 $\mathfrak{P}^A_{n}(F\otimes a)=\mathfrak{P}_{n}(F)\|a\|_A$.
 
 Last, we introduce $C_{u}^\infty(\k^{2d}, A)$ to be the subspace of smooth (in the sense of
 Bruhat) elements in 
 $C_{u}(\k^{2d}, A)$ for the regular representations $\tau\otimes\Id$ of $\k^{2d}$
 (see \cite{Meyer} for more details):
 $$
 C_{u}^\infty(\k^{2d}, A):=\big\{F\in C_{u}(\k^{2d}, A)\,:\,\tilde\tau(F):=[X\mapsto \tau_X(F)]\in \CE\big(\k^{2d},
 C_{u}(\k^{2d}, A)\big)\big\}.
 $$

\begin{lem}
\label{lem2}
 The space $\CB(\k^{2d}, A)$ is  Fr\'echet and 
 $ C_{u}^\infty(\k^{2d}, A)\subset \CB(\k^{2d}, A)\subset C_{u}(\k^{2d}, A)$ with dense inclusions.
\end{lem}
\begin{proof}
That $\CB(\k^{2d}, A)$ is  Fr\'echet follows from standard arguments.

To prove that $\CB(\k^{2d},A)\subset C_{u}(\k^{2d},A)$, we assume first that $A=\C$.
Since $L^1(\k^{2d})\ast L^\infty(\k^{2d})= C_{u}(\k^{2d})$ (see for instance \cite[(32.45) (b),
p. 283]{HR2}), it suffices to show that $\CB(\k^{2d})\subset L^1(\k^{2d})\ast L^\infty(\k^{2d})$.
So, let $F\in\CB(\k^{2d})$ and set $G:=J^{2p+1}F$. We have   $F=J^{-2p-1} G= \CG(\mu_0^{-2p-1})
\ast G$, which is the desired factorization. Indeed,  $G\in\CB(\k^{2d})\subset
L^\infty(\k^{2d})$ and $\CG(\mu_0^{-2p-1})\in L^1(\k^{2d})$ because
$\mu_0^{-2p-1}\in L^1(\k^{2d})$ and because  $\CG(\mu_0^{-2p-1})$
is compactly supported by Lemma \ref{lem1} (i).
For a generic $C^*$-algebra $A$, we deduce that the algebraic tensor product $\CB(\k^{2d})\otimes A$
is contained in $C_{u}(\k^{2d},A)$.
Since $\CB(\k^{2d})\otimes A$ is dense in $\CB(\k^{2d},A)$ and since 
$\mathfrak P_0^A$ is the $C^*$-norm
of $C_{u}(\k^{2d},A)$, 
we conclude that $\CB(\k^{2d},A)$ is contained in the norm closure of $\CB(\k^{2d})\otimes A$ in 
 $C_{u}(\k^{2d},A)$.

Next, for $F\in C_{u}(\k^{2d},A)$ and $\vf\in\CS(\k^{2d})$, we set
\begin{align}
\label{approx}
\tau_\vf(F):=\int_{\k^{2d}} \vf(X)\,\tau_X(F)\,dX,
\end{align}
where $\tau_X$ is the operator of translation by $X\in\k^{2d}$. By
 isometry of $\tau_X$ for the norm $\mathfrak P_0^A$, we get
 $\mathfrak P_0^A(\tau_\vf F)\leq \|\vf\|_1\mathfrak P_0^A(F)$. Hence, the integral in \eqref{approx}
 converges in $C_{u}(\k^{2d},A)$ since the latter is closed in $C_{b}(\k^{2d},A)$.
The operator $J$ commuting  with $\tau_X$, we get for all $n\in\N$
 $J^n \tau_\vf(F)=\tau_{J^n \vf}(F)$ which entails that  
 $\mathfrak P_n^A(\tau_\vf(F))\leq \|J^n\vf\|_1\mathfrak P_0^A(F)$. Hence, $\tau_\vf(F)\in\CB(\k^{2d},A)$.
 Chose next a positive sequence $\{\vf_k\}_{k\in\N}$ in $\CS(\k^{2d})$ such that  $\|\vf_k\|_1=1$
 and such that $\vf_k$ is supported in $B(0,k^{-1})$, the open ball centered at $0$ of radius $k^{-1}$.
 Then we have
 $$
 F-\tau_{\vf_k}(F)=\int_{\k^{2d}}\vf_k(X)\big(F-\tau_X(F)\big) dX,
 $$
 which entails that
 $$
 \mathfrak P_0^A(F-\tau_{\vf_k}(F))\leq\sup_{X\in B(0,k^{-1})} \mathfrak P_0^A\big(F-\tau_X(F)\big)=
 \sup_{X\in B(0,k^{-1})}\sup_{Y\in\k^{2d}}\|F(Y)-F(Y-X)\|_A,
 $$
 which goes to zero when $k$ goes to infinity due to the uniform continuity of $F$.
In particular, the set of finite sums of elements of the
form $\tau_\vf(F)$, $\vf\in\CS(\k^{2d})$, $F\in C_{u}(\k^{2d},A)$ is dense in $\CB(\k^{2d},A)$.
Since $\CD(\k^{2d})$ is dense in $\CS(\k^{2d})$, we deduce that the set of finite sums of elements of the
form $\tau_\vf(F)$, $\vf\in\CD(\k^{2d})$, $F\in C_{u}(\k^{2d},A)$ is also dense in $\CB(\k^{2d},A)$.
But by the extension of the Dixmier-Malliavin theorem for arbitrary locally compact groups, as
stated in \cite[Theorem 4.16]{Meyer}, the former space coincides with $ C_{u}^\infty(\k^{2d}, A)$.
Hence $C_{u}^\infty(\k^{2d}, A)$ is a dense subspace of $\CB(\k^{2d},A)$ but since 
$C_{u}^\infty(\k^{2d}, A)$ is also dense in $C_u(\k^{2d}, A)$, we get that $\CB(\k^{2d},A)$
is  dense in $C_u(\k^{2d}, A)$ too.
\end{proof}

\begin{rmk}
Observe that 
$C_{u}^\infty(\k^{2d}, A)=C_{u}(\k^{2d}, A)\cap \CE(\k^{2d},A)$. Since an element
in $\CB(\k^{2d},A)$ does not need to be locally constant, the dense inclusion 
$C_{u}^\infty(\k^{2d}, A)\subset \CB(\k^{2d},A)$ is proper.
\end{rmk}

Next, we come to the crucial fact that $\CB(\k^{2d}, A)$ is stable under point-wise
multiplication, which contrary to the case of $ C_{u}^\infty(\k^{2d}, A)$, is not obvious
at all. This essentially follows from the integral representation of elements in $\CB(\k^{2d},A)$:
\begin{lem}
\label{IRL}
Let $n\in\N$ and $F\in\CB(\k^{2d},A)$. Then, for all $N\geq n+2d+1$, we have
the uniformly (in $X\in\k^{2d}$) absolutely convergent integral representation:
\begin{align}
\label{IR}
J^n F(X)=|2|_\k^{2d}\int_{\k^{2d} \times \k^{2d}}\overline\Psi\big(2[Y,Z]\big)\mu_0^n(Y-Z)
\mu_0^{-N}(Y)\mu_0^{-N}(Z)\,\big(J^{N}F\big)(Y+X)\,dY\,dZ.
\end{align}
\end{lem}
\begin{proof}
Thanks to the Peetre inequality,  the integral on the right hand side of
\eqref{IR} is absolutely convergent in $A$ and the convergence is  uniform
 in $X\in\k^{2d}$ as it should be.
Assume first that the result is proven for $n=0$. 
The invariance of the Haar measure by translation gives then:
$$
F(X)=|2|_\k^{2d}\int_{\k^{2d} \times \k^{2d}}\overline\Psi\big(2[Y-X,Z-X]\big)
\mu_0^{-N}(Y-X)\mu_0^{-N}(Z-X)\big(J^{N}F\big)(Y)\,dYdZ.
$$
Applying $J^n$ on both sides  when $N\geq n+2d+1$, we get from Lemma
\eqref{lem1} (ii) and since $J$ commutes with $I$ and with the translations:
\begin{align*}
J^nF(X)&=|2|_\k^{2d}\int_{\k^{2d} \times \k^{2d}}\!\!\overline\Psi\big(2[Y-X,Z-X]\big)
\mu_0^n(Y-Z)
\mu_0^{-N}(Y-X)\mu_0^{-N}(Z-X)\big(J^{N}F\big)(Y)\,dYdZ,
\end{align*}
which gives the result.
Hence, it is enough to prove the result for $n=0$. In this case, the statement is immediate
 for $F\in\CS(\k^{2d},A)$: 
 Define $S_F(X)$ to be the right hand side of \eqref{IR} for $n=0$.
  Since $\mu_0^{-N}\in L^1(\k^{2d})\cap L^\infty(\k^{2d})$, 
 it also belongs to $L^2(\k^{2d})$ and
$S_F(X)$  can be rewritten as
\begin{align*}
S_F(X)&=\big\langle\CG\big(\mu_0^{-N}\big), I^{-N} J^N\tau_X(F)\big\rangle=
\big\langle\CG\big(\mu_0^{-N}\big),  J^N\tau_X(F)\big\rangle\\
&=\big\langle\mu_0^{-N},\mu_0^{N}  \CG\big(\tau_X(F)\big)\big\rangle=\CG \CG\big(\tau_X(F)\big)(0)=
\tau_X(F)(0)=F(X),
\end{align*}
where the second  equality follows by Lemma \ref{lem1} (i), 
 third equality follows by Plancherel
and the last three are immediate. Now, the general case of $F\in\CB(\k^{2d},A)$ follows easily 
by duality: Taking $\vf\in\CS(\k^{2d},A)$ arbitrary, one sees by Fubini that $\langle S_F,\vf\rangle
=\langle F,\overline{S_{\bar\vf}}\rangle$ which (from the preceding case) reads
$\langle F,\vf\rangle$. Identifying  $\CB(\k^{2d},A)$ with a subspace of $\CS'(\k^{2d},A)$, we are done.
\end{proof}

\begin{rmk}
Form a slight modification of the arguments given in the above lemma, we also deduce:
\begin{align}
\label{IGR}
 F=|2|_\k^{2d}\int_{\k^{2d} \times \k^{2d}}\overline\Psi\big(2[Y,Z]\big)\,\mu_0^{-2d-1}(Y)\mu_0^{-2d-1}(Z)\,
 \tau_Y\big(J^{2d+1}F\big)\,dY\,dZ\,,\quad\forall F\in\CB(\k^{2d},A),
\end{align}
where now the integral is absolutely convergent for all the seminorms of $\CB(\k^{2d},A)$.
\end{rmk}

\begin{cor}
\label{cor3}
$\CB(\k^{2d}, A)$ is a Fr\'echet algebra under the point-wise product. More precisely, for 
all  $n\in\N$ and all $ F_1,F_2\in\CB(\k^{2d},A)$ we have:
$$
\mathfrak P_{n}^A(F_1F_2)
  \leq  \|\mu_0^{-2d-1}\|_1^4\, \,\mathfrak P_{n+2d+1}^A(F_1)\,
 \mathfrak P_{n+2d+1}^A(F_2).
$$ 
\end{cor}
\begin{proof}
Fix $n\in\N$. Using Lemma \ref{IRL} twice, we get for $N= n+2d+1$:
\begin{align*}
&F_1F_2(X)=|2|_\k^{4d}\int\overline\Psi\big(2[X,Y_1-Z_1+Y_2-Z_2]\big)
\overline\Psi\big(2[Y_1,Z_1]\big)\overline\Psi\big(2[Y_2,Z_2]\big)\\
&\quad\times\mu_0^{-N}(Y_1-X)\mu^{-N}_0(Z_1-X)\mu^{-N}_0(Y_2-X)\mu_0^{-N}(Z_2-X)
\big(J^{N}F\big)(Y_1)\big(J^{N}F\big)(Y_2)\,dY_1dZ_1dY_2dZ_2.
\end{align*}
Applying $J^n$ on both sides, we deduce since $J$ commutes with the operator of multiplication
by $\mu_0$ and by its translates:
\begin{align}
\label{PPP}
J^n\big(F_1F_2\big)(X)&=|2|_\k^{4d}\int\overline\Psi\big(2[X,Y_1-Z_1+Y_2-Z_2]\big)
\overline\Psi\big(2[Y_1,Z_1]\big)\overline\Psi\big(2[Y_2,Z_2]\big)
\nonumber\\
&\quad\times\mu_0^n(Y_1-Z_1+Y_2-Z_2)
\mu_0^{-N}(Y_1-X)\mu_0^{-N}(Z_1-X)\mu^{-N}_0(Y_2-X)\mu^{-N}_0(Z_2-X)\nonumber\\
&\quad\times\big(J^{N}F\big)(Y_1)\big(J^{N}F\big)(Y_2)\,dY_1dZ_1dY_2dZ_2.
\end{align}
One concludes using the Peetre inequality together with  $|2|_\k\leq 1$.
\end{proof}

\begin{rmks}
Since $I^n(fF)=(I^nf)F$,  for $f\in\CS(\k^{2d},A)$ and $F\in\CB(\k^{2d},A)$ we also get from \eqref{PPP}
$$
\mathfrak P^A_{m,n}(fF)\leq \|\mu_0^{-2d-1}\|_1^4\, \mathfrak P^A_{m,n+2d+1}(f)\,
\mathfrak P^A_{n+2d+1}(F).
$$
Hence $\CS(\k^{2d},A)$ is an ideal of $\CB(\k^{2d},A)$ for the point-wise product.
\end{rmks}

Last, we need to prove that the space $\CB(\k^{2d},A)$ behaves well under  certain dilations. 
For this, we
need to introduce some more notations. For $\theta\in\k$, we let $D_\theta$ be the operator
of dilation by $\theta$: $D_\theta F(X):=F(\theta X)$.    
 Also, we let $I_\theta$ to be the operator 
 of multiplication by $D_\theta\mu_0$ and $J_\theta:=\CG I_\theta\CG$.
 Note that for $\theta=0$, $I_\theta=J_\theta=\Id$.
 \begin{lem}
 \label{DILL}
Let $\theta \in \CO_\k$ and retain the notations given above. \\
(i) As operators on $\CS'(\k^{2d},A)$, we have
 $$
 [I_\theta,J]=[I,J_\theta]=0.
 $$
(ii) The operator $J_\theta^n$, $n\in\N$, maps continuously $\CB(\k^{2d},A)$ to $C_{b}(\k^{2d},A)$
 with
 $$
  \mathfrak P_0^A(J_\theta^n F)\leq\|\mu_0^{-2d-1}\|_1^2 \,
 \mathfrak P_{n+2d+1}^A(F).
 $$
(iii) We have $J^n D_\theta=D_\theta J_\theta^n$
and consequently, the operator $D_\theta$ is continuous on $\CB(\k^{2d},A)$ with
 $$
 \mathfrak P_n^A(D_\theta F)\leq  \|\mu_0^{-2d-1}\|_1^2  \,
 \mathfrak P_{n+2d+1}^A(F).
 $$
  \end{lem} 
 \begin{proof}
  (i) The vanishing of the first commutator  follows because when $\theta\in\CO_\k$,  $D_\theta  
  \mu_0$ is also
 invariant by translations in  $(\frac12\CO_\k)^d\times (\frac12\CO_\k^o)^d$. The vanishing
 of the second commutator follows from the first, after conjugation by the symplectic Fourier
 transform.
 
 (ii) A minor adaptation of Lemma \ref{IRL},  which uses a minor adaptation of Lemma \ref{lem1} (ii)
 and (i), entails that for $N=n+2d+1$:
 $$
 J_\theta^n F(X)=|2|_\k^{2d}\int_{\k^{2d} \times \k^{2d}}\overline\Psi\big(2[Y,Z]\big)
 \mu_0^n(\theta Y-\theta Z)
\mu_0^{-N}(Y)\mu^{-N}_0(Z)\,\big(J^{N}F\big)(Y+X)\,dY\,dZ.
$$
The estimate then follows from the Peetre inequality together with the estimate $\mu_0(\theta X)
\leq\mu_0(X)$, valid when $|\theta|_\k\leq 1$.

(iii) The equality $J^n D_\theta=D_\theta J_\theta^n$ follows by direct computation and implies
the last inequality from the one obtained in (ii).
 \end{proof}
 \begin{rmk}
The lemma above is false for $\theta\in\k\setminus\CO_\k$. This is
 the (technical) reason why we have restricted the range of the deformation parameter to be  $\CO_\k$. 
 \end{rmk}

\begin{dfn}
\label{Areg}
The space $A_{\rm reg}$ of regular elements  in 
$A$ for the action $\alpha$ is given by:
\begin{align}
\label{A0}
 A_{\rm reg}:=\big\{a\in A\,:\,\tilde\alpha(a)\in\CB(\k^{2d}, A)\big\},
\end{align}
where the map $\tilde\alpha:A\to C_{u}(\k^{2d},A)$ is described in \eqref{alphamap}.
\end{dfn}
We  endow $ A_{\rm reg}$
with the topology associated with the transported seminorms:
\begin{align}
\label{SNn}
\|.\|_n^A: A_{\rm reg}\to\R_+,\qquad a\mapsto \mathfrak P_n^A\big(\tilde\alpha(a)\big)\,,\quad n\in\N.
\end{align}
Observe that $A_{\rm reg}$ depends on the action $\alpha$. When we need to stress this dependence, 
we will denote the space of regular elements by $A_{\rm reg}^\alpha$.

We   also need the space $A^\infty$, consisting in smooth vectors of $A$  in the  sense
of Bruhat, as considered in \cite{Meyer}:
\begin{align}
\label{Ainfty}
 A^\infty:=\big\{a\in A\,:\,\tilde\alpha(a)\in\CE(\k^{2d}, A)\big\}.
\end{align}
Since $\k^{2d}$ is totally disconnected, $\tilde\alpha(a)\in \CE(\k^{2d})$ if and only if it 
 is locally constant.  Hence, an element $a\in A$ belongs to $A^\infty$
if and only if there exists an open neighborhood $U$ of $0$ in $\k^{2d}$ such that
for all $x\in U$, we have $\alpha_x(a)=a$.
As expected, we have:
\begin{prop}
\label{DAA}
$A_{\rm reg}$ is a dense and  $\alpha$-stable Fr\'echet subalgebra of $A$. 
Moreover the action $\alpha$ is isometric for each seminorm \eqref{SNn} and
$A^\infty\subset A_{\rm reg}$ with a dense inclusion.
\end{prop}
\begin{proof}
$A_{\rm reg}$ is clearly a linear subspace of $A$. Moreover, by Corollary \ref{cor3}, we have
for all $a,b\in A_{\rm reg}$:
\begin{align*}
\|ab\|_n^A=\mathfrak P_{n}^A\big(\tilde\alpha(ab)\big)&=\mathfrak P_{n}^A\big(\tilde\alpha(a)
\tilde\alpha(b)\big)
 \\& \leq\|\mu_0^{-2d-1}\|_1^4\, \mathfrak P_{n+2d+1}^A\big(\tilde\alpha(a)\big)\,
 \mathfrak P_{n+2d+1}^A\big(\tilde\alpha(b)\big)= \|\mu_0^{-2d-1}\|_1^4\,\|a\|^A_{n+2d+1}\|b\|^A_{n+2d+1},
 \end{align*}
 hence $A_{\rm reg}$ is an algebra. Let now $X\in\k^{2d}$ and $a\in A_{\rm reg}$. Since $J$ commutes
 with translations, we have:
 $$
 \|\alpha_X(a)\|_n^A=\mathfrak P_{n}^A\big(\tilde\alpha\big(\alpha_X(a)\big)\big)=
\mathfrak P_{n}^A\big( \tau_{X}\big(\tilde\alpha(a)\big)\big)=
\mathfrak P_{n}^A\big( \tilde\alpha(a)\big)=\|a\|_n^A.
 $$
 Hence $\alpha$ is isometric for each seminorm $\|.\|_n^A$ and thus $A_{\rm reg}$ is preserved by 
 $\alpha$. The restriction to $A_{\rm reg}$
  of the map $\tilde\alpha:A\to C_{u}(\k^{2d},A)$ identifies $A_{\rm reg}$
 with a closed subspace of $\CB(\k^{2d},A)$. Since the topology of $A_{\rm reg}$ is inherited from those
 of $\CB(\k^{2d},A)$ via this identification, $A_{\rm reg}$ is Fr\'echet. That $A_{\rm reg}$ is dense in $A$ follows
 from an argument almost identical to those of Lemma \ref{lem2} : by considering for every $a\in A$
 the sequence in $A_{\rm reg}$ given by $\alpha_{\vf_k}(a):=\int_{\k^{2d}} \vf_k(X)\,\alpha_X(a)\,dX$, 
 where $0\leq\vf_k\in\CS(\k^{2d})$  has integral one and support within $B(0,k^{-1})$.
 Finally, that $A^\infty$ is dense in $A_{\rm reg}$ follows by the Dixmier-Malliavin Theorem \cite[Theorem 
 4.16]{Meyer} which shows that $A^\infty$ coincides with the finite linear sums of elements of the form
 $\alpha_\vf(a)$, $\vf\in\CD(\k^{2d})$ and $a\in A$, which 
  is dense in the set of  finite linear sums of elements of the form
 $\alpha_\vf(a)$, $\vf\in\CS(\k^{2d})$ and $a\in A$.
\end{proof}

\subsection{The deformed product}

Our goal is to give a meaning to the formula \eqref{SP} on $A_{\rm reg}$.  
Since $\tilde\alpha:A_{\rm reg}\to\CB(\k^{2d},A)$ is a continuous (indeed isometric for each seminorm)
 embedding of Fr\'echet spaces,  we will first 
work on $\CB(\k^{2d},A)$ and then pull back our results to $A_{\rm reg}$. Until soon, that $A$ carries
an action of $\k^{2d}$ is unimportant. Let $K(X,Y):=\overline\Psi(2[X,Y])$.
Seen as an element of $\CS'(\k^{2d}\times\k^{2d})$, the content of Lemma \ref{lem1} (ii) is that
$$
J\otimes I^{-1} K=K \quad\mbox{and}\quad I^{-1}\otimes J K=K .
$$
Hence, using further the commutativity of $I$ and $J$, we find for all $N\in\N$:
$$
K=(I^{-N}\otimes J^{N})(J^{N}\otimes I^{-N})K=(J^{N}\otimes J^{N})(I^{-N}\otimes I^{-N}) K=
J^{N}\otimes J^{N}\big(\mu_0^{-N}\otimes\mu_0^{-N} )K.
$$
In particular, for $F\in\CS(\k^{2d}\times\k^{2d},A)$, we get the equality for all $N\in\N$:
\begin{align}
\label{OT}
&\int_{\k^{2d}\times\k^{2d}} \overline{\Psi}\big(2[Y,Z]\big) F(X,Y)\,dX\,dY=\nonumber\\
&\qquad\qquad\quad\quad\quad\int_{\k^{2d}\times\k^{2d}} \overline{\Psi}\big(2[Y,Z]\big)\,
\mu_0^{-N}(Y)\,\mu_0^{-N}(Z)\,\big(J^N\otimes J^N \,F\big)(X,Y)\,dX\,dY.
\end{align}
The point is that since  $\mu_0^{-N}\in L^1(\k^{2d})$, $N\geq 2d+1$, 
the right hand side of \eqref{OT} still makes sense for $F\in \CB(\k^{2d}\times\k^{2d},A)$
when $N$ is large enough.
In the following, we refer to the identity \eqref{OT} as the oscillatory trick.

For  $F\in\CB(\k^{2d},A)$, we observe that the map 
$\tilde\tau(F):=\big[(X,Y)\in\k^{2d}\times\k^{2d}\mapsto \big(\tau_X F\big)(Y)\in A\big]$,
belongs to $ \CB(\k^{2d}\times\k^{2d},A)$ and that 
\begin{align}
\label{ID}
\tilde\tau\big( J^s F)=J^s\otimes\Id\,\tilde\tau(F)\,,\quad s\in\R.
\end{align}
The oscillatory trick \eqref{OT}, Lemma  \ref{DILL} (iii)  and the equality \eqref{ID}   
 suggest to  extend the star-product $\star_\theta$ from $\CS(\k^{2d})$ to
 $\CB(\k^{2d},A)$ as follows:
\begin{prop}
\label{prop1}
Let $\theta\in\CO_\k$. Then the bilinear map
\begin{align}
\label{Star}
\star_\theta&: \CB(\k^{2d},A)\times \CB(\k^{2d},A)\to \CB(\k^{2d},A)\;,\nonumber\\
(F_1,F_2)&\mapsto |2|_\k^{2d}\int_{\k^{2d} \times \k^{2d}} \overline\Psi\big(2[Y,Z]\big)\,
\mu_0^{-2d-1}(Y)\,
\mu_0^{-2d-1}(Z)\,\tau_{\theta Y}\big(J_\theta^{2d+1}F_1\big)\,\tau_Z\big(J^{2d+1}F_2\big)\,dY\,dZ,
\end{align}
is continuous and associative. Moreover when $\theta=0$, we have $F_1\star_{\theta=0} F_2=F_1 \,F_2$.
\end{prop}
\begin{proof}
For $F_1,F_2\in \CB(\k^{2d},A)$ and $n\in\N$, we have
\begin{align*}
J^n(F_1\star_\theta F_2)= |2|_\k^{2d}
\int_{\k^{2d}\times\k^{2d}} \overline\Psi\big(2[Y,Z]\big)\,
\mu_0^{-2d-1}(Y)\,\mu_0^{-2d-1}(Z)\,J^n\big(\tau_{\theta Y}\big(J_\theta^{2d+1}F_1\big)\,
\tau_Z\big(J^{2d+1}F_2\big)\big)
\,dY\,dZ.
\end{align*}
Hence we get
$$
\mathfrak P^A_n\big(F_1\star_\theta F_2\big)\leq
\|\mu_0^{-2d-1}\|_1^2 \sup_{Y,Z\in\k^{2d}} 
\mathfrak P^A_n\big(\tau_{\theta Y}\big(J_\theta^{2d+1}F_1\big)\,\tau_Z\big(J^{2d+1}F_2\big)\big).
$$
Therefore, by Corollary \ref{cor3} and Lemma \ref{DILL} (ii) (and the fact that $J$ and $J_\theta$
commute), we deduce 
\begin{align}
\label{IN1}
\mathfrak P^A_n\big(F_1\star_\theta F_2\big)&\leq \,
\|\mu_0^{-2d-1}\|_1^6 \sup_{Y,Z\in\k^{2d}} 
\mathfrak P^A_{2d+1+n}\big(\tau_{\theta Y}\big(J_\theta^{2d+1}F_1\big)\big)\,
\mathfrak P^A_{2d+1+n}\big(\tau_Z\big(J^{2d+1}F_2\big)\big)
\\&= \, \|\mu_0^{-2d-1}\|_1^6\,
\mathfrak P^A_{2d+1+n}\big(J_\theta^{2d+1}F_1\big)\,
\mathfrak P^A_{2d+1+n}\big(J^{2d+1}F_2\big)\nonumber\\
&= \, \|\mu_0^{-2d-1}\|_1^8\,
\mathfrak P^A_{6d+3+n}\big(F_1\big)\,
\mathfrak P^A_{4d+2+n}\big(F_2\big),\nonumber
\end{align}
which  proves continuity. 

Associativity is obvious when $A=\C$: it is the shadow of  the associativity of the algebra of bounded
operators on $L^2(\k^d)$
(see \cite[Th\'eor\`eme 3.3]{Bechata} from which it follows that the quantization map $\bO_\theta^\C:
\CB(\k^{2d})\to\CB\big(L^2(\k^d)\big)$ is injective). It immediately implies the associativity at the
level of the algebraic tensor product  $\CB(\k^{2d})\otimes A$. We conclude by density of the former 
in $\CB(\k^{2d},A)$: Fix $\eps>0$ and $n\in\N$. For $F_j\in\CB(\k^{2d},A)$, we let 
$F_j^\eps\in\CB(\k^{2d})\otimes A$ be
such that $\mathfrak P_k^A(F_j-F_j^\eps)\leq\eps$ for any $j=1,2,3$, and $k\in\{ 6d+3+n,
8d+4+n,10d+5+n,12d+6+n\}$. Then we get
\begin{align*}
&F_1\star_\theta (F_2\star_\theta F_3)-(F_1\star_\theta F_2)\star_\theta F_3
=\\
&\qquad(F_1-F^\eps_1)\star_\theta(F_2\star_\theta F_3)
+F_1^\eps\star_\theta \big((F_2-F_2^\eps)\star_\theta F_3\big)
+F_1^\eps\star_\theta \big(F_2^\eps\star_\theta (F_3-F_3^\eps)\big)\\
&\qquad-\big((F_1-F_1^\eps)\star_\theta F_2\big)\star_\theta F_3
-\big(F_1^\eps\star_\theta (F_2-F_2^\eps)\big)\star_\theta F_3
-(F_1^\eps\star_\theta F_2^\eps)\star_\theta( F_3-F_3^\eps),
\end{align*}
and  from \eqref{IN1}:
\begin{align*}
&\|\mu_0^{-2d-1}\|_1^{-16}\mathfrak P^A_n\big(F_1\star_\theta (F_2\star_\theta F_3)-(F_1\star_\theta F_2)\star_\theta F_3\big)\\
&\leq 
\mathfrak P^A_{6d+3+n}(F_1-F^\eps_1)\mathfrak P^A_{10d+5+n}(F_2)\mathfrak P^A_{8d+4+n}(F_3)
+\mathfrak P^A_{6d+3+n}(F_1^\eps) \mathfrak P^A_{10d+5+n}(F_2-F_2^\eps)\mathfrak P^A_{8d+4+n}
( F_3)\\
&+\mathfrak P^A_{6d+3+n}(F_1^\eps)\mathfrak P^A_{10d+5+n}(F_2^\eps)
\mathfrak P^A_{8d+4+n} (F_3-F_3^\eps)
+\mathfrak P^A_{12d+6+n}(F_1-F_1^\eps)\mathfrak P^A_{10d+5+n}( F_2\big)\mathfrak P^A_{8d+4+n}
( F_3)\\
&+\mathfrak P^A_{12d+6+n}(F_1^\eps)\mathfrak P^A_{10d+5+n}(F_2-F_2^\eps)\big)
\mathfrak P^A_{8d+4+n} (F_3)
+\mathfrak P^A_{12d+6+n} (F_1^\eps) \mathfrak P^A_{10d+5+n} ( F_2^\eps)\mathfrak P^A_{8d+4+n} 
(F_3-F_3^\eps)\\
&\leq \eps\Big(\mathfrak P^A_{10d+5+n}(F_2)\mathfrak P^A_{8d+4+n}(F_3)
+\mathfrak P^A_{6d+3+n}(F_1^\eps)\mathfrak P^A_{8d+4+n}
( F_3)
+\mathfrak P^A_{6d+3+n}(F_1^\eps)\mathfrak P^A_{10d+5+n}(F_2^\eps)\\
&+\mathfrak P^A_{10d+5+n}( F_2\big)\mathfrak P^A_{8d+4+n}
( F_3)
+\mathfrak P^A_{12d+6+n}(F_1^\eps)
\mathfrak P^A_{8d+4+n} (F_3)+
\mathfrak P^A_{12d+6+n} (F_1^\eps) \mathfrak P^A_{10d+5+n} ( F_2^\eps)\Big).
\end{align*}
Using last $\mathfrak P^A_{k} (F_j^\eps)\leq \eps+\mathfrak P^A_{k} (F_j)$ ($j=1,2,3$, $k\in\{ 6d+3+n,
8d+4+n,10d+5+n,12d+6+n\}$), we deduce that for all $n\in\N$, 
$F_1\star_\theta (F_2\star_\theta F_3)-(F_1\star_\theta F_2)\star_\theta F_3$ can be rendered as small
as one wishes in the seminorms $\mathfrak P^A_n$, hence this associator vanishes.

The fact that the deformed product coincides with the point-wise product when $\theta=0$
follows directly from Lemma \ref{IRL}.
\end{proof}
\begin{rmk}
\label{morepower}
 Obviously, we have 
$$
F_1\star_\theta F_2=|2|_\k^{2d}
\int_{\k^{2d}\times\k^{2d}} \overline\Psi\big(2[Y,Z]\big)\,
\mu_0^{-N}(Y)\,\mu^{-N}_0(Z)\,\tau_{\theta Y}\big(J_\theta^{N}F_1\big)\,\tau_Z\big(J^{N}F_2\big)\,dY\,dZ,
$$
for any $N\in\N$ such that $N\geq 2d+1$. Using moreover the commutation of $J$ with translations,
we also deduce the point-wise expression:
$$
F_1\star_\theta F_2(X) = |2|_\k^{2d}\int_{\k^{2d} \times \k^{2d}} \overline\Psi\big(2[Y,Z]\big)\,
\mu_0^{-N}(Y)\,\mu^{-N}_0(Z)\,\big(J_\theta^{N}F_1\big)(X+\theta Y)\,\big(J^{N}F_2\big)(X+Z)\,dY\,dZ.
$$
Last, when  $F_1,F_2\in \CS(\k^{2d},A)$, we can undo the oscillatory trick to get:
$$
F_1\star_\theta F_2(X) = |2|_\k^{2d}\int_{\k^{2d} \times \k^{2d}} \overline\Psi\big(2[Y,Z]\big)\,
F_1(X+\theta Y)\,F_2(X+Z)\,dY\,dZ.
$$
\end{rmk}
The following representation of the product $\star_\theta$ will be useful to handle the deformed product
in a rather simple way.
\begin{lem}
\label{sequence}
Let $\theta\in\CO_\k$. 
For $F_1,F_2\in\CB(\k^{2d},A)$ and $N\in\N$, set
$$
F_N:=|2|_\k^{2d} \int_{\k^{2d}\times\k^{2d}} \overline\Psi\big(2[Y,Z]\big)\,
\tau_{\theta Y}(F_1)\,\tau_{Z}( F_2)\,e^{-\mu_0(Y)\mu_0(Z)/N}\,dY\,dZ.
$$
Then, the sequence $(F_N)_{N\in\N}$ belongs to $\CB(\k^{2d},A)$ and 
converges to $F_1\star_\theta F_2$ for the topology of $\CB(\k^{2d},A)$.
\end{lem}
\begin{proof}
That $F_N$, $N\in\N$, belongs to $\CB(\k^{2d},A)$ follows from arguments almost identical
to those given in the first part of the proof of Proposition \ref{prop1}. Next, 
using the oscillatory trick together with the commutativity of $I$ and $J$, we get
$$
F_N=|2|_\k^{2d} \int \overline\Psi\big(2[Y,Z]\big)\,
\mu_0^{-2d-2}(Y)\,\mu_0^{-2d-2}(Z)\,\tau_{\theta Y}\big(J_\theta^{2d+2}F_1\big)\,\tau_Z\big(J^{2d+2}F_2\big)
\,e^{-\mu_0(Y)\mu_0(Z)/N}\,dY\,dZ,
$$
and thus (using  Remark \ref{morepower})
\begin{align*}
&F_1\star_\theta F_2-F_N=\\&
|2|_\k^{2d}\int\overline\Psi\big(2[Y,Z]\big)\,
\mu_0^{-2d-2}(Y)\,\mu_0^{-2d-2}(Z)\,\tau_{\theta Y}\big(J_\theta^{2d+2}F_1\big)\,\tau_Z\big(J^{2d+2}F_2\big)
\,\Big(1-e^{-\mu_0(Y)\mu_0(Z)/N}\Big)\,dY\,dZ.
\end{align*}
Using Corollary  \ref{cor3} and Lemma \ref{DILL} (ii), we then deduce
$$
\mathfrak P_n^A(F_1\star_\theta F_2-F_N)\leq \|\mu_0^{-2d-1}\|_1^8\,
\mathfrak P_{n+6d+4}^A(F_1) \mathfrak P_{n+4d+3}^A(F_2)
\sup_{Y,Z\in\k^{2d}}\frac{1-e^{-\mu_0(Y)\mu_0(Z)/N}}{\mu_0(Y)\mu_0(Z)}.
$$
Observing then that
$$
\sup_{Y,Z\in\k^{2d}}\frac{1-e^{-\mu_0(Y)\mu_0(Z)/N}}{\mu_0(Y)\mu_0(Z)}\leq 
\sup_{x>0}\frac{1-e^{-x/N}}x\leq \frac1N ,
$$
we get the result.
\end{proof}

We also note:
\begin{lem}
\label{SIB}
Let $\theta\in\CO_\k$. Then,
  $\big(\CS(\k^{2d},A),\star_\theta\big)$ is an ideal of $\big(\CB(\k^{2d},A),\star_\theta\big)$. 
\end{lem}
\begin{proof}
Let  $f\in\CS(\k^{2d},A)$, $F\in\CB(\k^{2d},A)$ and $n,m\in\N$.
For $M,N$  arbitrary integers satisfying $M,N\geq 2d+1$. By Remark \ref{morepower}
we have
$$
\mathfrak P^A_{m,n}\big(f\star_\theta F\big)\leq\int_{\k^{2d}\times\k^{2d}} 
\mu_0(Y)^{-N}\,\mu_0(Z)^{-M}
\mathfrak P^A_{m,n}\big(\tau_{\theta Y}\big(J_\theta^M f\big)\,\tau_Z\big(J^NF\big)\big)
\,dY\,dZ.
$$
Using Corollary \ref{cor3} and Lemma \ref{DILL}  again and  the Peetre inequality, we deduce
\begin{align*}
\mathfrak P^A_{m,n}\big(\tau_{\theta Y}\big(J_\theta^M f\big)\,\tau_Z\big(J^NF\big)\big)&=
\mathfrak P^A_{n}\big(\mu_0^m\tau_{\theta Y}\big(J^M_\theta f\big)\,\tau_Z\big(J^NF\big)\big)\\
&\leq
\|\mu_0^{-2d-1}\|_1^4 \,
\mathfrak P^A_{2d+1+n}\big(\mu_0^m\tau_{\theta Y}\big(J^M_\theta f\big)\big)\,\mathfrak P^A_{2d+1+n}\big(\tau_Z\big(J^NF\big)\big)\\
&\leq \|\mu_0^{-2d-1}\|_1^6\, \mu_0^m(Y)\,
\mathfrak P^A_{m,4d+2+n+M}(f)\,\mathfrak P^A_{2d+1+n+N}(F).
\end{align*}
Choosing $M=2d+1$ and $N=2d+1+m$, we deduce
$$
\mathfrak P^A_{m,n}\big(f\star_\theta F\big)\leq \|\mu_0^{-2d-1}\|_1^8 \, 
\mathfrak P^A_{m,6d+3+n}(f)\,\mathfrak P^A_{4d+2+m+n}(F).
$$
The case of $F\star_\theta f$ is similar. 
\end{proof}

\begin{lem}
\label{invol}
Let $\theta\in\CO_\k$. 
With $*$  the involution of $A$, we set $F^*(X):=F(X)^*$.
Then we have $(F_1\star_\theta F_2)^*=F_2^*\star_\theta F_1^*$ for all $F_1,F_2\in\CB(\k^{2d},A)$.
\end{lem}
\begin{proof}
Observe that the involution defined on $\CB(\k^{2d},A)$ 
 is continuous and commutes with $J$ and $\tau$. Therefore, we get from Lemma \ref{sequence}
\begin{align*}
(F_1\star_\theta F_2 )^*&= \lim_{N\to\infty}|2|_\k^{2d}\int_{\k^{2d}\times\k^{2d}}{\Psi}\big(2[Y,Z]\big)\,
e^{-\mu_0(Y)\mu_0(Z)/N}\,\big(\tau_{\theta Y} (F_1)\,\tau_Z(F_2)\big)^*
\,dY\,dZ\\
&=  \lim_{N\to\infty}|2|_\k^{2d} \int_{\k^{2d}\times\k^{2d}} \overline\Psi\big(2[Y,Z]\big)\,
e^{-\mu_0(\theta Y)\mu_0(\theta^{-1}Z)/N}\tau_{\theta Y}( F_2^* )\,\tau_{Z}(F_1^* )dY
\,dZ. 
\end{align*}
But from the same reasoning that the one given in Lemma \ref{sequence}, and
using Lemma \ref{DILL} (i) for the commutativity of $I_\theta$ and $J$,
one sees that the expression
above is exactly $F_2^*\star_\theta F_1^*$.
\end{proof}
\begin{lem}
\label{aut}
Let $\theta\in\CO_\k$. 
The action of $\k^{2d}$ by translation on $\CB(\k^{2d},A)$  is still an automorphism for
the deformed product $\star_\theta$.
\end{lem}
\begin{proof}
This follows from the defining relation \eqref{Star} of $\star_\theta$ on $\CB(\k^{2d},A)$ together with the fact
that $\tau$ is  continuous  and commutes with $J$ on $\CB(\k^{2d},A)$. 
\end{proof}

\begin{dfn}
The deformed product of the Fr\'echet algebra $A_{\rm reg}$ is given by the map:
$$
\star_\theta^\alpha :A_{\rm reg}\times A_{\rm reg}\to A,\quad (a,b)\mapsto\tilde\alpha(a)
\star_\theta\tilde\alpha(b) (0),
$$
which by Remark \ref{morepower} can be rewritten as:
\begin{align}
\label{OM}
a\star_\theta^\alpha b= |2|_\k^{2d} \int \overline\Psi(2[X,Y])\,\mu_0^{-2d-1}
(X)\,\mu_0^{-2d-1}(Y)\,\big(J_\theta^{2d+1}\tilde\alpha(a)\big)(\theta X) \,\big(J^{2d+1}
\tilde\alpha(b)\big)(Y)\,dXdY.
\end{align}

\end{dfn}

We arrive to our first main result.

\begin{thm}
\label{TDF}
Let $\k$ be a non-Archimedean local field of characteristic different from $2$ and $\theta\in\CO_\k$. Let
also $A$ be a $C^*$-algebra endowed with a continuous action $\alpha$
of $\k^{2d}$. Then (keeping the notations displayed above) 
$(A_{\rm reg},\star_\theta^\alpha)$ is an associative  Fr\'echet algebra that we call the 
 Fr\'echet deformation of the $C^*$-algebra $A$. Moreover, 
 the original action $\alpha$ is still by automorphisms and the original involution 
 is still an involution.
\end{thm}
\begin{proof}
By construction, the action $\alpha$ yields an isometric embedding of 
$A_{\rm reg}$ in  $\CB(\k^{2d},A)$,
 Proposition \ref{prop1} entails then that $\star_\theta:\CB(\k^{2d},A)\times\CB(\k^{2d},A)\to \CB(\k^{2d},A)$
 continuously. Note also that the evaluation map $\CB(\k^{2d},A)\to A$, $F\mapsto F(0)$ is continuous too. Hence,
$\star_\theta^\alpha:A_{\rm reg}\times A_{\rm reg}\to A$ is continuous
and   from the inequality  \eqref{IN1}, we deduce
$$
\|a\star_\theta^\alpha b\|_A\leq\|\mu_0^{-2d-1}\|_1^2\, \|a\|^A_{2d+1}\,\|b\|^A_{2d+1}.
$$
Next, we need to show that the map $\star_\theta^\alpha$ takes values in $A_{\rm reg}$ 
(and not only in $A$). To show this, let  $a\in A_{\rm reg}$ and $X,Y\in\k^{2d}$. Observe first that $\tau_X \circ\tilde\alpha(a)=
\tilde\alpha(\alpha_{X}(a))$.
Consider then the action $\hat\alpha$ of $\k^{2d}$ on $\CB(\k^{2d},A)$ given by 
$\big(\hat\alpha_X(F)\big)(Y)=\alpha_X\big(F(Y)\big)$. Then we have 
$\hat\alpha_X\big(\tilde\alpha(a)\big)=\tau_{X}\big(\tilde\alpha(a)\big)$. Since $\tau$ commutes with $J$,
we therefore get $\hat\alpha_X\big(J^n\tilde\alpha(a)\big)=\tau_{X}\big(J^n\tilde\alpha(a)\big)$,
 from which we 
easily deduce by \eqref{OM} that the map $\tilde\alpha$ intertwines $\star_\theta^\alpha$
and $\star_\theta$:
\begin{align}
\label{OOO}
\tilde\alpha(a\star_\theta^\alpha b)=\tilde\alpha(a)\star_\theta
\tilde\alpha(b)\,,\quad\forall a,b\in A_{\rm reg}.
\end{align}
Eq. \eqref{OOO} immediately implies that  $\star_\theta^\alpha$ takes values in $A_{\rm reg}$.
Moreover, it also implies  the associativity of $\star_\theta^\alpha$ on $A_{\rm reg}$ from the associativity of 
$\star_\theta$
on $\CB(\k^{2d},A)$:
$$
(a\star_\theta^\alpha b)\star_\theta^\alpha c=\tilde\alpha(a\star_\theta^\alpha b)\star_\theta 
\tilde\alpha(c)(0)=\big(\tilde\alpha(a)\star_\theta\tilde\alpha(b)\big)
\star_\theta\tilde\alpha(c)(0)\,,\quad\forall a,b,c\in A_{\rm reg}.
$$
Observe then that \eqref{OOO} (together with Lemma \ref{aut})
also implies that the action $\alpha$
on $A_{\rm reg}$ is still by automorphism of the deformed product:
\begin{align*}
\alpha_X(a\star_\theta^\alpha b)&=\tilde\alpha(a)\star_\theta\tilde\alpha(b)(X)
=\tau_{X}\big(\tilde\alpha(a)\star_\theta\tilde\alpha(b)\big)(0)
=\big(\tau_{X}\tilde\alpha(a)\star_\theta\tau_{X}\tilde\alpha(b)\big)(0)\\&
=\tilde\alpha(\alpha_X(a))\star_\theta\tilde\alpha(\alpha_X(b))(0)=
\alpha_X(a)\star_\theta^\alpha\alpha_X(b).
\end{align*}
Last, that the original involution is still an involution follows from Lemma \ref{invol}.
\end{proof}

\begin{rmk}
\label{C2}
Theorem \ref{TDF}  can be extended in two directions. Firstly, if $\k$ is of characteristic $2$,
then all the statements of this section (including  Theorem  \ref{TDF})
continue to hold true provided we redefine the function $\mu_0$ in \eqref{mu}, the symplectic Fourier
transform $\CG$ in \eqref{SFT} and the deformed product $\star_\theta$ in \eqref{Star} without the factor $2$.
However, and as indicated earlier, we then lose the contact with the pseudo-differential calculus
that we will intensively use in the next section in order to construct a $C^*$-norm on the deformed
Fr\'echet 
algebra $(A_{\rm reg},\star_\theta^\alpha)$. Secondly,  it is not difficult to extend Theorem  \ref{TDF} 
in the case of a Fr\'echet algebra $\CA$ (instead of a $C^*$-algebra $A$). If the topology of $\CA$ comes 
from a countable set of seminorms $\{\|.\|_j\}_{j\in\N}$, then we only need to require that it carries a 
 continuous action $\alpha$ of $\k^{2d}$ which is tempered in the sense that for all $j\in\N$,
there exist $C>0$ and $k,n\in\N$ such that for all $a\in\CA$, $\|\alpha_X(a)\|_j\leq C\mu_0(X)^{n}
\|a\|_k$.
But again, to construct a deformed $C^*$-norm we have to restrict ourselves to $C^*$-algebras and isometric
actions.
\end{rmk}

\begin{rmk}
By equivariance of the deformed product and from the discussion which follows \eqref{Ainfty},
we deduce that
$A^\infty$ is also stable under $\star_\theta^\alpha$. However it is not clear if we have continuity for the
topology of $A^\infty$.
\end{rmk}

\section{The $C^*$-deformation of a $C^*$-algebra}
\label{CDCA}
\subsection{The Wigner functions approach}
\label{1}
In this section,  we  assume that $\k$ is of characteristic different from
$2$ and that $\theta\in\CO_\k\setminus\{0\}$.
Also, we identify our $C^*$-algebra $A$ with a subalgebra of $\CB(\CH)$
for a separable Hilbert space $\CH$.

By analogy with the integral representation  \eqref{SF},
 we may define for  $f\in L^1(\k^{2d},A)$:
\begin{align}
\label{SFA}
\bO_\theta^A(f):=\big|\tfrac2\theta\big|^{d}_\k\int_{\k^{2d}} \Omega_\theta(X)\otimes f(X)\,dX.
\end{align}
The  map $\bO_\theta^A$ sends continuously $L^1(\k^{2d},A)$ to
$\CB(L^2(\k^d)\otimes \CH)$. Indeed, since 
$$
\|\Omega_\theta(X)\|_{\CB(L^2(\k^d))}=\|U_\theta(X)\Sigma U_\theta(X)^*\|_{\CB(L^2(\k^d))}=1\,,
\quad\forall X\in\K^{2d},
$$ 
we get
\begin{align}
\label{may}
\|\bO_\theta^A(f)\|_{\CB(L^2(\k^d)\otimes \CH)}
&\leq |\theta|^{-d}_\k\int_{\k^{2d}} \|\Omega_\theta(X)\otimes f(X)\|_{\CB(L^2(\k^d)\otimes_\CH)}\,dX\\
&\quad=
 |\theta|^{-d}_\k\int_{\k^{2d}} \| f(X)\|_A\,dX=
 |\theta|^{-d}_\k\,\|f\|_1.\nonumber
\end{align}
Since moreover $\Omega_\theta(X)$ is selfadjoint, we get $\bO_\theta^A(f)^*=\bO_\theta^A(f^*)$, where 
$f^*\in L^1(\k^{2d},A)$ is defined by $f^*(X):=f(X)^*$.
There is an obvious reason to introduce the map $\Omega_\theta^A$:
\begin{lem}
\label{mor}
 The map $\bO_\theta^A:\big(\CS(\k^{2d},A),\star_\theta\big)\to\CB\big(L^2(\k^d)\otimes\CH\big)$
is a continuous $*$-homomorphism.
\end{lem}
\begin{proof}
That $\bO_\theta^A$ is continuous and involution preserving has already been proved.
That $\bO_\theta^A$ is a homomorphism when $A=\C$ follows by construction of the $\star_\theta$.
 Hence, $\bO_\theta^A$ is still a homomorphism
 at the level of the algebraic tensor product $\CS(\k^{2d})\otimes A$. 
For $j=1,2$, take $f_j\in\CS(\k^{2d},A)$ and choose $(f_{j,k})_{k\in\N}\subset\CS(\k^{2d})\otimes A$
converging to $f_j$ in the topology of $\CS(\k^{2d},A)$. 
Then, we have in $\CB(L^2(\k^d)\otimes \CH)$:
\begin{align*}
\|\bO_\theta^A(f_1)\bO_\theta^A(f_2)-\bO_\theta^A(f_1\star_\theta f_2)\|&\leq \|\bO_\theta^A(f_1-f_{1,k})\bO_\theta^A(f_2)\|+
\|\bO_\theta^A(f_{1,k})\bO_\theta^A(f_2-f_{2,k})\|\\&+\|\bO_\theta^A((f_1-f_{1,k})\star_\theta f_2)\|+
\|\bO_\theta^A(f_{1,k}\star_\theta (f_2-f_{2,k}))\|,
\end{align*}
which by the estimates \eqref{IN1}, \eqref{may} and $\|f\|_1\leq \|\mu_0^{-2d-1}\|_1\,
\mathfrak P^A_{2d+1,0}(f)  $, may be rendered as small as wished 
by choosing $k\in\N$ large enough. Hence $\|\bO_\theta^A(f_1)\bO_\theta^A(f_2)-\bO_\theta^A(f_1\star_\theta f_2)\|=0$
and thus $\bO_\theta^A(f_1)\bO_\theta^A(f_2)=\bO_\theta^A(f_1\star_\theta f_2)$.
\end{proof}

Let now  $\eta$ be the characteristic function of $\CO_\k^d$, normalized by $\|\eta\|_2 = 1$ and
let also $W^\theta_{X,Y}$ the Wigner function associated with the pair $(\eta^\theta_X,\eta^\theta_Y)$ of
coherent states as in \eqref{Wigner2}. For $F\in\CB(\k^{2d},A)$, we can then define the 
following $A$-valued  function on $\k^{2d}\times\k^{2d}$:
\begin{align}
\label{W(F)}
W_{X,Y}^{\theta,A}(F):=\big|\tfrac2\theta\big|^{d}_\k\int_{\k^{2d}} W^\theta_{X,Y}(Z)\,F(Z)\,dZ.
\end{align}
\begin{lem}
\label{W(F)L}
Let $F\in\CB(\k^{2d},A)$. Then for all $X,Y\in\k^{2d}$ and all $n\in\N$, we have:
$$
\|W^{\theta,A}_{X,Y}(F)\|_A\leq \mu_0^{-n} \big(\tfrac1{2\theta}(X -Y )\big)\,\mathfrak P^A_n(F).
$$
If moreover $f\in\CS(\k^{2d},A)$ then for all $m,n\in\N$ we have
$$
\|W^{\theta,A}_{X,Y}(f)\|_A\leq  \mu_0^{-n} \big(\tfrac12(X+Y)\big)
 \mu_0^{-m} \big(\tfrac1{2\theta}(X-Y)\big) \,
\mathfrak P^A_{n,m}(f).
$$
\end{lem}
\begin{proof}
Note that if $F\in\CB(\k^{2d},A)$, we have
$$
W_{X,Y}^{\theta,A}(F)=\big|\tfrac2\theta\big|^{d}_\k\int_{\k^{2d}} \big(J^{-n}W^\theta_{X,Y}\big)(Z)\,\big(J^nf\big)(Z)\,dZ,\quad\forall n\in\N,
$$
and if $f\in\CS(\k^{2d},A)$, we have
$$
W_{X,Y}^{\theta,A}(f)=\big|\tfrac2\theta\big|^{d}_\k\int_{\k^{2d}} \big(I^{-n}J^{-m}W^\theta_{X,Y}\big)(Z)\,\big(I^nJ^mf\big)(Z)\,dZ,\quad\forall m,n\in\N.
$$
The result follows immediately from Lemma \ref{Bech}.
\end{proof}

Now, for $\phi,\psi\in L^2(\k^d)$, we denote by $|\phi\rangle\langle\psi|$,
the rank one operator $L^2(\k^d)\to L^2(\k^d)$, $\rho\mapsto\langle \psi,\rho\rangle\phi$.
\begin{dfn} 
For $F\in\CB(\k^{2d},A)$, define in the weak sense in $L^2(\k^d)\otimes\CH$:
$$
{\bf W}_\theta^A(F):=|\theta|^{-2d}_\k\int_{\k^{2d}\times\k^{2d}} |\eta^\theta_Y\rangle\langle \eta^\theta_X|\otimes W^{\theta,A}_{X,Y}(F)\,dXdY.
$$ 
\end{dfn}
The following  may be thought as a variant of the Calderon-Vaillancourt Theorem for 
our $A$-valued Weyl pseudo-differential calculus on local fields. This result
for $A=\C$ is due to Bechata  \cite{Bechata}.
\begin{prop}
\label{CV}
The quadratic form associated with the weak integral operator ${\bf W}_\theta^A(F)$ defines an
element of $\CB(L^2(\k^d)\otimes\CH)$, with
$$
\|{\bf W}_\theta^A(F)\|_{\CB(L^2(\k^d)\otimes\CH)}\leq \|\mu_0^{-2d-1}\|_1\, \mathfrak P^A_{2d+1}(F).
$$
Moreover, we have ${\bf W}_\theta^A(F)^*={\bf W}_\theta^A(F^*)$, where $F^*\in\CB(\k^{2d},A)$ is defined
by $F^*(X):=F(X)^*$.
\end{prop}
\begin{proof}
For $\Phi\in L^2(\k^d)\otimes\CH$ and $\vf\in L^2(\k^d)$ we denote by 
$\langle \vf,\Phi\rangle_{L^2(\k^d)}$ the vector in $\CH$ defined by
$\langle \langle\vf, \Phi\rangle_{L^2(\k^d)},\rho\rangle_\CH:
=\langle\Phi,\vf\otimes\rho\rangle_{L^2(\k^d)\otimes\CH}$
for every $\rho\in\CH$. With this in mind and with $\eta$  the characteristic function of $(\CO_\k)^d$ 
(normalized by $\|\eta\|_2 = 1$), it is not difficult to see that the resolution of the identity 
\eqref{RI} on $L^2(\k^d)$ entails:
\begin{align}
\label{RI2}
\|\Phi\|^2_{L^2(\k^d)\otimes\CH}
=|\theta|_\k^{-d}\int_{\k^{2d}} \|\langle \eta^\theta_X,\Phi\rangle_{L^2(\k^d)}\|_\CH^2\,dX\,,
\quad \forall \Phi\in L^2(\k^d)\otimes\CH.
\end{align}
Take now $\Phi_1,\Phi_2\in L^2(\k^d)\otimes\CH$. We therefore get
\begin{align*}
\big|\big\langle \Phi_1,{\bf W}_\theta^A(F)\Phi_2\big\rangle_{L^2(\k^d)\otimes\CH}\big|\leq |\theta|^{-3d}_\k
\int_{\k^{2d}\times\k^{2d}} \|\langle \eta^\theta_Y,\Phi_1\rangle_{L^2(\k^d)}\|_\CH\,
\|\langle \eta^\theta_X,\Phi_2\rangle_{L^2(\k^d)}\|_\CH\,\|W^A_{X,Y}(F)\|\,dX\,dY.
\end{align*}
By the Cauchy-Schwarz inequality and \eqref{RI2}, we deduce
that the integral above is bounded by
$$
 |\theta|^{-2d}_\k
\|\Phi_1\|_{L^2(\k^d)\otimes\CH}\|\Phi_2\|_{L^2(\k^d)\otimes\CH}
 \Big(\sup_{X\in\k^{2d}}\int_{\k^{2d}}\|W^{\theta,A}_{X,Y}(F)\|_A \,dY \Big)^{1/2}
\Big(\sup_{Y\in\k^{2d}}\int_{\k^{2d}}\|W^{\theta,A}_{X,Y}(F)\|_A \,dX \Big)^{1/2},
$$
Hence, by Lemma \ref{W(F)L} we get
$$
\big|\big\langle \Phi_1,{\bf W}_\theta^{\theta,A}(F)\Phi_2\big\rangle_{L^2(\k^d)\otimes\CH}\big|\leq
\|\Phi_1\|_{L^2(\k^d)\otimes\CH}\|\Phi_2\|_{L^2(\k^d)\otimes\CH}\, \|\mu_0^{-2d-1}\|_1\,
\mathfrak P^A_{2d+1}(F),
$$
which completes the proof.
\end{proof}
\begin{rmk}
Observe that the bound of the norm of ${\bf W}_\theta^{\theta,A}(F)$ we have obtained, is independent of 
parameter $\theta$. 
\end{rmk}
\begin{cor}
\label{HW}
The map ${\bf W}_\theta^A:\big(\CB(\k^{2d},A),\star_\theta\big)\to\CB\big(L^2(\k^d)\otimes\CH\big)$
is a continuous $*$-homomorphism which extends $\bO_\theta^A:\big(\CS(\k^{2d},A),\star_\theta\big)\to
\CB\big(L^2(\k^d)\otimes\CH\big)$.
\end{cor}
\begin{proof}
By Proposition \ref{CV}, ${\bf W}_\theta^A$ is continuous and involution preserving. 
When $A=\C$, the relations ${\bf W}_\theta^A(F_1){\bf W}_\theta^A(F_2)={\bf W}_\theta^A(F_1\star_\theta F_2)$
and  $\bO_\theta^A(f)={\bf W}_\theta^A(f)$, for 
 $f\in\CS(\k^{2d})$ and $F,F_1,F_2\in\CB(\k^{2d})$
are implicit  in the work of Bechata \cite{Bechata} (they are almost
tautological). Obviously, these relations are still valid at the level of algebraic tensor products. 
The general case follows from
the same methods as those used in Lemma \ref{mor}, using the estimates \eqref{IN1},
\eqref{may} and Proposition \ref{CV}.
\end{proof}
\begin{rmk}
Since $\CS(\k^{2d},A)$ is an ideal in $\CB(\k^{2d},A)$, we deduce from  Corollary \ref{HW} that
$\bO_\theta^A(f\star_\theta F)=\bO_\theta^A(f){\bf W}_\theta^A(F)$ whenever  
$f\in\CS(\k^{2d},A)$ and $F\in\CB(\k^{2d},A)$. 
\end{rmk}

We are now able to state the main result of this section, whose proof is an immediate consequence
of Corollary \ref{HW}.
\begin{thm}
Let $\k$ be a non-Archimedean local field of characteristic different from $2$, let
$\theta\in\CO_\k\setminus\{0\}$ and let
 $A\subset\CB(\CH)$ be a $C^*$-algebra endowed with a  continuous action $\alpha$
of $\k^{2d}$. Then (keeping the notations displayed above) the norm:
$$
A_{\rm reg}\to\R^+\,,\quad a\mapsto\|a\|_\theta:=
\big\|{\bf W}_\theta^A\big(\tilde\alpha(a)\big)\big\|_{\CB(L^2(\k^d)\otimes\CH)},
$$ 
endows
$(A_{\rm reg},\star_\theta^\alpha,*)$ with the structure of a pre-$C^*$-algebra. 
We call its completion the 
$C^*$-deformation of the $C^*$-algebra $A$ and denote it by $A_\theta^\alpha$ (or by $A_\theta$ when
no confusion can occur). 
\end{thm}

We first mention  an important feature, namely that the deformed algebra $A_\theta$ still carries
a continuous action of $\k^{2d}$. We stress that this property heavily  relies on the fact that 
our group is Abelian. For non-Abelian groups (e.g$.$ \cite{BG,NT}), the only 
surviving action is the one of a quantum group.
\begin{prop}
\label{ac}
The action $\alpha$ of $\k^{2d}$ on $(A_{\rm reg},\star_\theta^\alpha)$ extends to a  continuous 
action on $A_\theta$, that we denote by  $\alpha_\theta$.
\end{prop}
\begin{proof}
Once we will have shown that $\alpha$ gives a  continuous action of $\k^{2d}$ on the Fr\'echet algebra $(A_{\rm reg},\star_\theta^\alpha)$,
the existence of the extension $\alpha_\theta$
  can be easily proven following the lines of  \cite[Proposition 5.11]{Ri}. That $\alpha_\theta$
 is  continuous on $A_{\rm reg}$ follows from the fact that the isometric embedding $A_{\rm reg}\to \CB(\k^{2d},A)$ 
 intertwines $\alpha$ with $\tau$ and that $\tau$ is  continuous on $\CB(\k^{2d},A)$
 as the latter is a subspace of $C_{u}(\k^{2d},A)$. That it is by automorphism on $A_\theta$ follows 
 from  Lemma \ref{aut}.
\end{proof}

\subsection{The $C^*$-module approach}
\label{2}
We  now  realize the deformed $C^*$-norm $\|.\|_\theta$ as  the  $C^*$-norm of bounded adjointable 
endomorphisms of a $C^*$-module for $A$,
 in a  manner very similar to the $\R^d$-case \cite{Ri}. However, this construction cannot substitute  the previous one since lattice methods
used in \cite{Ri} are not available here. 
In fact, we are in the same situation than those of negatively curved
K\"ahlerian Lie groups \cite{BG}.

Let
 $\langle.,.\rangle_A$  be the $A$-valued sesquilinear pairing on $\CS(\k^{2d},A)$ given by
$$
\langle f_1,f_2\rangle_A:=\int_{\k^{2d}} f_1^* (X) f_2(X)\,dX.
$$
This paring  is clearly well defined. Testing this paring on elementary
tensors, we deduce (by the density of products in a $C^*$-algebra) that $\langle\CS(\k^{2d},A),
\CS(\k^{2d},A)\rangle_A$ is dense in $A$. It is manifestly positive since 
$\langle f,f\rangle_A=\int_{\k^{2d}} |f (X)|^2\,dX\geq 0$ and 
$\langle f_1,f_2\rangle_A^*=\langle f_2,f_1\rangle_A$. If we endow further $\CS(\k^{2d},A)$
with the right action of the undeformed $C^*$-algebra $A$ given by juxtaposition:
$f.a:=[X\mapsto f(X)a]$, then we get $\langle f_1,f_2a\rangle_A=\langle f_1,f_2\rangle_A a$.
Hence, $\CS(\k^{2d})$ becomes a right pre-$C^*$-module for the undeformed $C^*$-algebra $A$.
Now, by Lemma \ref{SIB}, we know that $\big(\CB(\k^{2d},A),\star_\theta\big)$ acts continuously on 
$\CS(\k^{2d},A)$ by 
$$
L_\theta(F):\CS(\k^{2d},A)\to\CS(\k^{2d},A)\,,\quad f\mapsto F\star_\theta f.
$$
This action clearly commutes with the right action of $A$. 
That $L_\theta(F)$ is also adjointable and bounded follows from the following 
alternative expression for the paring:
\begin{lem}
\label{above}
For $f_1,f_2\in\CS(\k^{2d},A)$ and $\theta\in\CO_\k\setminus\{0\}$, we have:
$$
\langle f_1,f_2\rangle_A= \int_{\k^{2d}}\langle\eta^\theta_X,\bO_\theta^A(f_1^*\star_\theta f_2)\eta^\theta_X\rangle_{L^2(\k^d)}\,dX.
$$
\end{lem}
\begin{proof}
By polarization,
we may assume without lost of generality that $f_1=f_2$.
Let $\langle .,.\rangle_A'$ be the $A$-valued paring given by the right hand side of 
 the equality we have to prove. Let us first show that it is well defined.  
Note that for $f\in\CS(\k^{2d},A)$, we have 
$$
\langle\eta^\theta_X,\bO_\theta^A(f^*\star_\theta f)\eta^\theta_X\rangle_{L^2(\k^d)}=
\big\langle\eta^\theta_X,\Big(\big|\tfrac2\theta\big|^{d}_\k
 \int_{\k^{2d}} f^*\star_\theta f(Y) \otimes\Omega_\theta(Y)dY\Big)\eta^\theta_X
\big\rangle_{L^2(\k^d)}.
$$
Since the integral  converges in the norm of $\CB\big(L^2(\k^d)\otimes\CH\big)$, we get
\begin{align*}
\langle\eta^\theta_X,\bO_\theta^A(f^*\star_\theta f)\eta^\theta_X\rangle_{L^2(\k^d)}&=\big|\tfrac2\theta\big|^{d}_\k
\int_{\k^{2d}} f^*\star_\theta f(Y) \langle\eta^\theta_X,\Omega_\theta(Y)\eta^\theta_X\rangle_{L^2(\k^d)} dY
\\&=
\big|\tfrac2\theta\big|^{d}_\k
\int_{\k^{2d}} f^*\star_\theta f(Y) \,W^\theta_{X,X}(Y) dY=
W_{X,X}^{\theta,A}( f^*\star_\theta f).
\end{align*}
We conclude using Lemma \ref{W(F)L} which gives in that case
$$
\|W^{\theta,A}_{X,X}( f^*\star_\theta f)\|_A\leq  \mu_0(X)^{-2d-1} \,\mathfrak P^A_{2d+1,0}(f^*\star_\theta f),
$$
and thus 
$$
\|\langle f,f\rangle_A'\|\leq   \|\mu_0^{-2d-1}\|_1 \,
\mathfrak P^A_{2d+1,0}(f^*\star_\theta f)<\infty.
$$
From this inequality, we also deduce that it is enough to treat the case $A=\C$. Indeed, if the equality
works on $\CS(\k^{2d})$ then it works on the algebraic tensor product  
$\CS(\k^{2d})\otimes A$ and one concludes using  a limiting argument based on
$\|\langle f,f\rangle_A\|\leq \|f\|_2^2$ and on  the inequality given above for
$\|\langle f,f\rangle_A'\|$.

In the case $A=\C$, note first that  by unitarity of the quantization map,
we have 
$$
\langle f,f\rangle_\C=|\theta|_\k^d\,\Tr\big({\bf \Omega_\theta}(f)^*{\bf \Omega_\theta}(f)\big)
=|\theta|_\k^d\,\Tr\big({\bf \Omega_\theta}( f^*\star_\theta f)\big).
$$ 
The resolution of the identity \eqref{RI}, implies that for a positive trace-class 
operator $S$ on $L^2(\k^d)$, we have 
$\Tr(S)=|\theta|_\k^{-d}\int_{\k^{2d}}\langle\eta^\theta_X,S\eta^\theta_X\rangle\,dX$. Indeed, since $\|\eta\|_2=1$, we get by \eqref{RI} 
that for all $\vf\in L^2(\k^d)$, $\langle\vf,\vf\rangle=|\theta|_\k^{-d}\int \langle\eta^\theta_X,\vf\rangle\langle\vf,\eta^\theta_X\rangle dX$.
Hence, for any orthonormal basis $(\vf_k)_{k\in\N}$,  using  monotone convergence and  
$\sum_{k\in\N}|\vf_k\rangle\langle\vf_k|=\Id$ in the weak sense, we get
\begin{align*}
\Tr(S)=\sum_{k\in\N}\langle\vf_k,S\vf_k\rangle_{L^2(\k^d)}
&=\sum_{k\in\N}\langle S^{1/2}\vf_k,S^{1/2}\vf_k\rangle_{L^2(\k^d)}\\
&=|\theta|_\k^{-d}\,\sum_{k\in\N}\int_{\k^{2d}} \langle S^{1/2}\vf_k,\eta^\theta_X\rangle_{L^2(\k^d)}
\langle\eta^\theta_X,S^{1/2}\vf_k\rangle_{L^2(\k^d)}
 dX\\
&=|\theta|_\k^{-d}\,\int_{\k^{2d}} \sum_{k\in\N} \langle S^{1/2}\eta^\theta_X,\vf_k\rangle_{L^2(\k^d)}\langle \vf_k,S^{1/2}\eta^\theta_X\rangle_{L^2(\k^d)}
 dX\\&
=|\theta|_\k^{-d}\,\int_{\k^{2d}}  \langle\eta^\theta_X,S\eta^\theta_X\rangle_{L^2(\k^d)} dX,
\end{align*}
which completes the proof.
\end{proof}

From the expression of $\langle.,.\rangle_A$ given in Lemma \ref{above}, 
it is clear that the operator $L_\theta(F)$, $F\in\CB(\k^{2d},A)$,
is adjointable with adjoint $L_\theta(F^*)$. But the elementary operator inequality
on $\CB(L^2(\k^d)\otimes\CH)$:
$$
\bO_\theta^A(f^*\star_\theta F^*\star_\theta F\star_\theta f)=\bO_\theta^A(f^*) |{\bf W}^A_\theta(F)|^2
\bO_\theta^A(f^*)\leq\| {\bf W}^A_\theta(F)\|^2 \bO_\theta^A(f^*\star_\theta f),
$$
entails  that $\langle L_\theta(F) f , L_\theta(F) f \rangle_A\leq \| {\bf W}^A_\theta(F)\|^2
\langle f , f \rangle_A$. Hence for $F\in\CB(\k^{2d},A)$,  $L_\theta(F)$ is 
bounded adjointable $A$-linear endomorphism of $\CS(\k^{2d},A)$,
 with $\| L_\theta(F)\|\leq \| {\bf W}^A_\theta(F)\|$.
In fact this inequality is an equality. 
Indeed, by construction, the restriction to the algebraic
tensor product $\CB(\k^{2d})\otimes A$ of the deformed $C^*$-norm $F\mapsto \| {\bf W}^A_\theta(F)\|$
coincides with the minimal $C^*$-norm on the algebraic tensor product of the $C^*$-completion
of $\big(\CB(\k^{2d}),\star_\theta\big)$ by $A$. But the restriction to the algebraic
tensor product $\CB(\k^{2d})\otimes A$ of the $C^*$-norm $F\mapsto \| L_\theta(F)\|$ extends to a
$C^*$-cross norm on the $C^*$-completion of $\big(\CB(\k^{2d}),\star_\theta\big)$ by $A$.
Hence, the two norms coincide. Restricting this to the image of $A_{\rm reg}$ in $\CB(\k^{2d},A)$,
we deduce:

\begin{prop}
\label{CMA}
Let $\theta\in\CO_\k\setminus\{0\}$.
Then the deformed $C^*$-norm $\|.\|_\theta$ on the Fr\'echet $*$-algebra 
$(A_{\rm reg},\star_\theta^\alpha,*)$ coincides with:
$$
A_{\rm reg}\to\R^+\,,\quad a\mapsto \|L_\theta\big(\tilde\alpha(a)\big)\|.
$$
\end{prop}

The main point with the realization of the deformed  $C^*$-norm as the operator norm on 
the  pre-$C^*$-module $\CS(\k^{2d},A)$ is that it still makes sense for the value $\theta=0$,
where there is no pseudo-differential calculus. Indeed, when $\theta=0$ the product on
$A_{\rm reg}$ is the undeformed one (by Proposition \ref{prop1}) and thus $A_{\rm reg}$ 
acts on the left of $\CS(\k^{2d},A)$ via $f\mapsto\tilde\alpha(a)\,f$. Since moreover
$$
\| L_{\theta=0}\big(\tilde\alpha(a)\big)\|=\sup_{X\in\k^{2d}}\|\alpha_X(a)\|_A=\|a\|_A\,,\quad
\forall a\in A_{\rm reg},
$$
we deduce that $A_{\theta=0}=A$.

As an illustration of the interest of the $C^*$-module approach to the deformation,
we  clarify the relations between the deformations of $C_0(\k^{2d},A)$ and
of $C_{u}(\k^{2d},A)$ for the action given by translations on the one hand and the $C^*$-closures of
the Fr\'echet algebras
$\big(\CS(\k^{2d},A),\star_\theta\big)$ and of $\big(\CB(\k^{2d},A),\star_\theta\big)$
induced by the representations $\bO_\theta^A$ and ${\bf W}^A_\theta$ on the other hand.

 \begin{prop}
 \label{normal}
 Let $\theta\in\CO_\k$.
 Consider the $C^*$-algebras $C_0(\k^{2d},A)$ and $C_{u}(\k^{2d},A)$ endowed with the
 action of $\k^{2d}$ given by $\tau\otimes\Id$ where as usual $\tau$ is the action by translations. Then
 we have the isomorphisms:
$$
 C_0(\k^{2d},A)_\theta\simeq \overline{(\CS(\k^{2d},A),\star_\theta)}\qquad
\mbox{ and}\qquad
 C_{u}(\k^{2d},A)_\theta\simeq \overline{(\CB(\k^{2d},A),\star_\theta)}.
$$
Moreover, when $\theta\in\CO_\k\setminus\{0\}$, then
$$
C_0(\k^{2d},A)_\theta\simeq 
 \mathcal K\big(L^2(\k^d)\big)\otimes A.
$$
 \end{prop}
\begin{proof}
 When $A=\C$ and $\theta\ne 0$,  the quantization map $\bO_\theta$  is a 
 (multiple of a) unitary 
operator  from $L^2(\k^{2d})$ to the Hilbert-Schmidt operators on $L^2(\k^d)$.   Since 
$\CS(\k^{2d})$ is densely contained in $L^2(\k^{2d})$, and since the Hilbert-Schmidt
operators are norm-dense in the compacts,
we get after completion $\overline{(\CS(\k^{2d}),\star_\theta)}\simeq  \mathcal K\big(L^2(\k^d)\big)$.
The associated isomorphism with $A$ arbitrary then follows by nuclearity of the 
compact operators.  The first two isomorphisms can be proven exactly as in
  \cite[Proposition 4.15]{Ri}, by observing that $\CB(\k^{2d},A)= C_{u}(\k^{2d},A)_{\rm reg}$
  and that $\CS(\k^{2d},A)\subset C_0(\k^{2d},A)_{\rm reg}$ densely, 
  and using the $C^*$-module picture for the deformed $C^*$-norm.
\end{proof}

In the case of actions of $\R^{2d}$, Rieffel proved in \cite{Ri3} that the $K$-theory is an invariant
of the deformation. In \cite{BG}, we also proved the same result for actions of negatively curved
K\"ahlerian group. From the isomorphisms given in Proposition \ref{normal},
we easily deduce this property no longer holds here:

\begin{cor}
The $K$-theory is not an invariant of the deformation.
\end{cor}
\begin{proof}
We  give a counter example. Take $A=C_0(\k^{2d})$. As $\k$ is a totally
disconnected space, $K_0(A)=C_c(\k^{2d},\mathbb Z)$. On the other hand, Proposition
\ref{normal} says  that the deformation of $A$ by the regular action is the $C^*$-algebra of compact operators.
Hence $K_0(A_\theta)=\Z$. (Note that in this example the $K_1$-group is not deformed
as it is trivial in both cases.)
\end{proof}

\subsection{The twisted crossed product approach}
\label{3}

There is a third way to realize the deformed $C^*$-norm on the deformed Fr\'echet algebra
$(A_{\rm reg},\star_\theta^\alpha)$, which is based on the work of 
Kasprzak \cite{K1} and its development   by  Neshveyev et al.  \cite{Nes-Shanga,NT,Nes-flat}.
Kasprzak's original construction uses general results on crossed product and the notion of Landstad 
algebras. It applies to 
 continuous actions of a locally compact Abelian groups on  $C^*$-algebras and is parametrized by a 
continuous unitary 2-cocycle 
on the dual group. In fact, the deformed algebra $A_\theta^K$ in Kasprzak's
picture is abstractly characterized by a crossed product bi-decomposition 
$\k^{2d}\ltimes_{\alpha_\theta} A_\theta^K=\k^{2d}\ltimes_{\alpha} A$, where $\alpha_\theta$
is the extension of $\alpha$ from $A_{\rm reg}$ to $A_\theta$, as described in Proposition \ref{ac}. 
 An equivalent and more concrete approach (see below) has been  given by 
Bhowmick, Neshveyev and  Sangha in \cite{Nes-Shanga}, which applies to
 continuous actions of  locally compact  groups (not necessarily Abelian)
 on $C^*$-algebras and is parametrized by a 
measurable unitary 2-cocycle on the dual, viewed as a quantum group. This approach to deformation had
been extended in full generality in \cite{NT} to continuous actions of  locally compact quantum  groups (in
the von Neumann algebraic setting)
on $C^*$-algebras and is still parametrized by a 
measurable unitary 2-cocycle on the dual quantum group. Here we mostly follow the paper
\cite{Nes-flat}, we let $\theta\in\CO_\k\setminus\{0\}$ and we still assume that $\k$ is of characteristic
different from $2$.

For $X\in\k^{2d}$, let  $V_X^\theta$ be  the unitary operator
on $L^2(\k^{2d})$ given  by
$$
V_X^\theta f(Y):=\overline\Psi\big(\tfrac2\theta[X,Y]\big)\, f(X+Y).
$$
The operators  $(V_X^\theta)_{X\in\k^{2d}}$ satisfies   Weyl type relations $V_{X+Y}^\theta
=\overline\Psi
(\tfrac2\theta[X,Y])\,V_X^\theta\,V_Y^\theta$. The $C^*$-sub-algebra of $\CB(L^2(\k^{2d}))$
 generated by the operators
$$
V_f^\theta:=\int_{\k^{2d}} f(X)\,V_X^\theta\,dX\,,\qquad f\in L^1(\k^{2d}),
$$
 is called the twisted group $C^*$-algebra and is denoted by  $C_\theta^*(\k^{2d})$.
For $f_1,f_2\in L^1(\k^{2d})$, we have $V_{f_1}^\theta V_{f_2}^\theta=V_{f_1\ast_\theta f_2}^\theta$,
 where $\ast_\theta$ is the twisted convolution product, defined by
 $$
 f_1\ast_\theta f_2(X):=\int_{\k^{2d}} f_1(X-Y)\,f_2(Y)\,\overline\Psi(\tfrac2\theta[Y,X])\,dY.
 $$
 Moreover, we have for $f_1,f_2\in\CS(\k^{2d})$:
 $$
 \CG_\theta\big( f_1\ast_\theta f_2\big)=\CG_\theta(f_1)\star_\theta\CG_\theta(f_2),
 $$
 where $\CG_\theta$ denotes the rescaled version of the symplectic Fourier transform:
 $$
( \CG_\theta f)(X):=\big|\tfrac2\theta\big|^d\int_{\k^{2d}} \Psi(\tfrac2\theta[Y,X])\,f(Y)\,dY.
$$
Hence  $C_\theta^*(\k^{2d})$ is isomorphic to $C_0(\k^{2d})_\theta$  thus
(by Proposition \ref{normal})
isomorphic to the $C^*$-algebra of compact operators. However, this does not implies directly
that the $2$-cocycle $(X,Y)\mapsto 
 \Psi(\tfrac2\theta[Y,X])$ on the selfdual group $\k^{2d}$ is regular in the sense of \cite[Definition 2.9]{NT}.
Since here the modular involution $\hat J$ is the complex conjugation, regularity here means 
 that $C^*_\theta(\k^{2d}) C_0(\k^{2d})\subset \CK(L^2(\k^{2d}))$. But this is a trivial fact here since
 the operator kernel of $V^\theta_{f_1}\,M_{f_2}$ ($M_f$ stands for the operator of point-wise 
 multiplication by $f$), for $f_1,f_2\in \CS(\k^{2d})$ is given by $\overline\Psi(2[X,Y])
 f_1(X-Y)f_2(Y)$ and so $V^\theta_{f_1}\,M_{f_2}$ is Hilbert-Schmidt thus compact, and one concludes
 by density.

Following \cite{Nes-flat}, we  introduce  the  operator on $\CS(\k^{2d},A_{\rm reg})$ given by:
\begin{align}
\label{T}
\Pi_\theta f(X):=|2|_\k^d\int_{\k^{2d}}\overline\Psi(2[X,Y])\,\alpha_{\theta Y}\big(\CG(f)(Y)\big)\,dY
\,,\quad\theta\in\CO_\k.
 \end{align}
 This operator is continuous. Indeed, $\Pi_\theta=\CG\circ S_\theta\circ\CG$, where 
 $S_\theta f(X):=\alpha_{\theta X}\big(f(X)\big)$, which is continuous on $\CS(\k^{2d},A_{\rm reg})$
(essentially by Lemma \ref{DILL}). Note however that $S_\theta$ does not need to be continuous on
  $\CS(\k^{2d},A)$. Since $S_\theta$ is invertible with inverse given by $S_{-\theta}$, $\Pi_\theta$
  is also invertible with inverse $\Pi_{-\theta}$.
   Defining the faithful representations $\pi$ and $\pi_\theta$ of the crossed products 
    $\k^{2d}\ltimes_{\alpha} A$ and
 $\k^{2d}\ltimes_{\alpha_\theta} A_\theta$  on the 
 Hilbert module $L^2(\k^{2d},A)$ given for $f,\xi\in \CS(\k^{2d},A_{\rm reg})$ by:
 $$
 \pi(f)\xi:=\int_{\k^{2d}} \tilde\alpha\big(f(X)\big) \,\big(\tau_X\, \xi)\,dX,\qquad
  \pi_\theta(f)\xi:=\int_{\k^{2d}} \tilde\alpha\big(f(X)\big) \star_\theta\big(\tau_X\, \xi)\,dX.
 $$

The first main result of \cite{Nes-flat}  is that the (Archimedean version
of the) map $\Pi_\theta$ extends to an isomorphisms of crossed products. This is 
the most important step to prove  equivalence between
Rieffel's and Kasprzak's approaches to deformation. 
A quick inspection shows that the proofs of \cite[Theorems 1.1 \& 2.1]{Nes-flat} 
extend  to our context without modification, this yields:

 \begin{prop}
 \label{NES1}
 For $f\in\CS(\k^{2d},A_{\rm reg})$, we have $\pi_\theta(f)=\pi\big(\Pi_\theta(f)\big)$. Moreover,
 $\Pi_\theta$ extends to an isomorphism of  crossed products 
 $\k^{2d}\ltimes_{\alpha_\theta} A_\theta\simeq\k^{2d}\ltimes_{\alpha} A$.
 \end{prop}
From this,
one deduces exactly as \cite[Theorem 3.10]{K1} that $A_\theta$ is nuclear if and only if $A$ is nuclear.
Mimicking  the arguments of \cite[Theorem 3.2]{Ri2} (see also
 \cite[Corollary 7.49]{BG}),  one can also prove that
our deformed $C^*$-algebra $A_\theta$ is strongly Morita equivalent to the crossed product
$$
\k^{2d}\ltimes_{\Ad(U_\theta)\otimes\alpha}\big(\CK(L^2(\k^d))\otimes A\big),
$$
 where  $U_\theta$ is the projective unitary irreducible representation of $\k^{2d}$ on  $L^2(\k^d)$
given in \eqref{SCH}.
This gives an alternative proof of the property of preservation of nuclearity,
which is  how Rieffel proved the analogous result for actions of $\R^{2d}$
 \cite[Theorem 4.1]{Ri2}. Note  that Proposition \ref{NES1} together with the Stone-von Neumann
 Theorem (see e.g$.$ \cite[Theorem C.34]{RW}) also implies that 
 the deformed $C^*$-algebra is stably isomorphic to a double crossed product of the undeformed
 $C^*$-algebra (see \cite[Theorem 3.6]{NT} for a more general
  statement in the context of regular cocycles
 for locally compact quantum groups):
$$
\CK\otimes A_\theta\simeq\k^{2d}\ltimes_{\beta_\theta}\big(\k^{2d}\ltimes_\alpha A\big),
$$
where $\beta_\theta$ is the image under $\Pi_\theta$ of the  action  dual to $\alpha_\theta$ of the selfdual
group $\k^{2d}$ on the crossed product $\k^{2d}\ltimes_{\alpha_\theta} A_\theta$   
(which is not 
the  action dual to $\alpha$  on  $\k^{2d}\ltimes_{\alpha} A$). 

The deformed $C^*$-algebra   constructed in \cite{Nes-Shanga} is based on the  following
``quantization maps'':
$$
T_{\theta,\nu}:C_0(\k^{2d})\to C^*_\theta(\k^{2d}), \quad f\mapsto \Id\otimes\nu\big(W\,\overline\Psi_\theta(
\CG_\theta\, f\,\CG_\theta\otimes\id)\Psi_\theta W^*\big).
$$
Here the function 
$f$ is viewed as an operator of multiplication on $L^2(\k^{2d})$, $\Psi_\theta$ is the operator 
of multiplication by the function $\big[(X,Y)\mapsto\Psi(\tfrac2\theta[X,Y])\big]$ on $L^2(\k^{2d}\times
\k^{2d})$ and $W$ is the multiplicative unitary on $L^2(\k^{2d}\times\k^{2d})$
given by $W\xi(X,Y)=\xi(X,Y-X)$. 
Last, $\nu$ is an element 
of the predual of the von Neumann algebra generated by $C^*_\theta(\k^{2d})$ in $L^2(\k^{2d})$.
Given that $C^*_\theta(\k^{2d})\simeq\CK(L^2(\k^{2d}))$, $\nu$ is 
 of the form $\Tr(A.)$, for $A$  of trace-class on $L^2(\k^{2d})$.
For $f\in\CS(\k^{2d})$,
a computation shows that (up to a constant) 
$$
T_{\theta,\nu}(f)=\int_{\k^{2d}}(\CG_\theta f)(X)\,\overline{\nu(V_{-X}^\theta)}\,V_{-X}^\theta\,dX.
$$
 From this, we can easily show that the union of the images of the maps $T_{\theta,\nu}$  is dense in $C^*_\theta(\k^{2d})$
 (see also \cite[Lemma 3.2]{NT} for a more general statement). To simplify the discussion we assume
 that $2$ is invertible in $\CO_{\k}$ (if not, the formulas are slightly different but the conclusion is 
 unchanged).
Let $\gamma\in\k$
 such that $|\gamma|_\k>1$, define for $n\in\Z$, $\vf_n:=\chi_{\theta\gamma^{n}(\CO_\k
 \times\CO_ \k^o)^d}\in\CS(\k^{2d})$ and consider the element $\nu_n:=
 \big|\tfrac{\gamma^{n}}\theta\big|^{2d}_\k\langle \vf_n,.\vf_{-n}\rangle$.
 Then we have
 $$
  \nu_n(V_{-X}^\theta)= \big|\tfrac{\gamma^{n}}\theta\big|^{2d}_\k\,\langle \vf_n,V_{-X}^\theta\vf_{-n}\rangle=
 \big|\tfrac{\gamma^{2n}}\theta\big|^{d}_\k\,\CG_\theta\big(
  \vf_n\,\tau_{-X}(\vf_{-n})\big)(X).
  $$
An explicit computation then shows that
  $ \nu_n(V_{-X}^\theta)=\vf_n(X)$ and thus
  $$
T_{\theta,\nu_n}(f)=\int_{\gamma^n\theta(\CO_\k\times\CO_
 \k^o)^d}(\CG_\theta f)(X)\,V_{-X}^\theta\,dX,
$$
which, by dominated convergence, converges to $V^\theta_{\CG_\theta (f)}$ when $n\to\infty$. One concludes using the fact that
$\CG_\theta$ is an automorphism of $\CS(\k^{2d})$ (indeed $\CG_\theta$ is an involution).

The crucial observation
in \cite{Nes-Shanga} is that for any $C^*$-algebra $A$, the map $T_{\theta,\nu}$ extends to a map
$$
T_{\theta,\nu}:M\big(C_0(\k^{2d})\otimes A\big)\to M\big(C^*_\theta(\k^{2d})\otimes A\big),
$$
which is continuous on the unit ball for the strict topology. Composing this map with $\tilde
\alpha:A\to C_u(\k^{2d},A)\subset M( C_0(\k^{2d},A))$, we get a family of maps 
$$
A\to M\big(C^*_\theta(\k^{2d})\otimes A\big),\quad a\mapsto
 \Id\otimes\nu\big(W\,\overline\Psi_\theta(
\CG_\theta\, \tilde\alpha(a)\,\CG_\theta\otimes\id)\Psi_\theta W^*\big).
$$
By definition, the deformation of $A$ in the sense \cite{Nes-Shanga} is the sub-$C^*$-algebra
$A_\theta^{BNS}$ of $M( C_0(\k^{2d},A))$
 generated  by the images of the maps $T_{\theta,\nu}\circ\tilde\alpha$.
The action $\hat\alpha_\theta:=[Y\mapsto\Ad(\CG_\theta\,\tau_Y\,\CG_\theta)]$ of $\k^{2d}$ on
 $A_\theta^{BNS}$ induces a representation of the crossed product $\k^{2d}\ltimes_{\hat\alpha_\theta}
 A_\theta^{BNS}$
 on the Hilbert module $L^2(\k^{2d},A)$. By \cite[Theorem 3.9]{NT} (see also the detailed discussion
 in \cite[pages 4-5]{Nes-flat}) we have  $\k^{2d}\ltimes_{\hat\alpha_\theta}A_\theta^{BNS}=\CG_\theta\,(
 \k^{2d}\ltimes_\alpha A)\,\CG_\theta$. Hence, $\CG_\theta\,\hat\alpha_\theta(A_\theta^{BNS})\,\CG_\theta
 \subset M( \k^{2d}\ltimes_\alpha A)$. Consider last the extension of the map \eqref{T}
 to the multipliers of crossed products:
 $$
\Pi_\theta:M\big( \k^{2d}\ltimes_{\alpha_\theta} A_\theta\big)\to M\big(\k^{2d}\ltimes_{\alpha} A\big).
$$
The proof of \cite[Theorem 2.3]{Nes-flat} (which is mainly based on general crossed product arguments)
 extends straightforwardly   to our context and gives a third way
to realize our deformed $C^*$-algebra:
\begin{thm}
\label{TT}
The map $\Pi_\theta$ establishes an isomorphism of 
$A_\theta\simeq\tilde\alpha_\theta(A_\theta)\subset 
M\big( \k^{2d}\ltimes_{\alpha_\theta} A_\theta\big)$ with
$A_\theta^{BNS}\simeq\CG_\theta\,\hat\alpha_\theta(A_\theta^{BNS})\,\CG_\theta
 \subset M( \k^{2d}\ltimes_\alpha A)$.
\end{thm}  

\section{Properties of the deformation}

In this final section we always assume that the characteristic of $\k$ is different from $2$
but  otherwise specified,
the deformation parameter $\theta$ can be freely chosen in $\CO_\k$ (i.e$.$
the value $\theta=0$ is also allowed).  Our aim is to show that most of the structural properties of the 
deformation survive in the non-Archimedean context. In order to give the shortest possible
proofs, we take advantages of the three
different ways to realize our deformed $C^*$-algebra: as a subalgebra of $\CB(L^2(\k^d))\otimes
_{\rm min} A$ as initially
 defined (see subsection \ref{1}), as a subalgebra of bounded adjointable $A$-linear
endomorphisms of the Hilbert module $L^2(\k^{2d},A)$ (see subsection \ref{2}) or as a subalgebra
of $M\big(C^*_\theta(\k^{2d})\otimes A\big)$ (see subsection \ref{3}). But it is important to mention 
that all the results that use the twisted group algebra approach can be alternatively proven by
 methods similar to those developed in \cite{Ri,Ri2}.

\quad

We first study the question of  approximate unit. This result is a very important technical tool in 
numerous forthcoming statements. Here, we have to substantially modify Rieffel's original arguments
and thus we provide a rather detailed proof (this is mostly due to the fact that the operator $J$ does not
satisfy any kind of Leibniz rule).
\begin{prop}
\label{BAU}
The deformed $C^*$-algebra $A_\theta$ possesses an approximate unit (in the sense of \cite{Pe})
consisting of elements of $A_{\rm reg}$. 
\end{prop} 
\begin{proof}
Let $\{e'_\lambda\}_{\lambda\in\Lambda}$ be a net of  approximate unit for $A$ ($e'_\lambda\in A_+$,
$\|e'_\lambda\|_A\leq 1$, $\lim_\lambda\|a-e'_\lambda a\|_A=0$, $\lim_\lambda\|a-ae'_\lambda \|_A=0$ for all 
$a\in A$)
 and let 
$0\leq\vf\in\CS(\k^{2d})$ with compact support and with $\int\vf=1$. Define then 
$$
e_\lambda:=\alpha_\vf(e'_\lambda)=\int_{\k^{2d}}\vf(X)\,\alpha_X(e'_\lambda)\,dX\in A_{\rm reg}.
$$
Since for all $a\in A$, we have
$$
\|a-e_\lambda a\|_A\leq \int_{\k^{2d}}\vf(X)\,\|a-\alpha_X(e'_\lambda)a\|_A\,dX=
\int_{\k^{2d}}\vf(X)\,\|\alpha_{-X}(a)-e'_\lambda\alpha_{-X}(a)\|_A\,dX,
$$
and similarly for $\|a-ae_\lambda \|$, a compactness argument over the support of $\vf$ entails that
$\{e_\lambda\}_{\lambda\in\Lambda}$ is an  approximate unit for the undeformed $C^*$-algebra
 $A$ consisting of elements of regular elements. 

We are going to prove  that $\{e_\lambda\}_{\lambda\in\Lambda}$ is also an approximate unit
for the deformed  $C^*$-algebra $A_\theta$. Since $A_{\rm reg}$ is dense in $A_\theta$, it suffices to prove that 
$\|a-a\star_\theta^\alpha e_\lambda \|_\theta$ and $\|a-e_\lambda\star_\theta^\alpha a\|_\theta$ go to zero for all $a\in A_{\rm reg}$.  By Proposition
\ref{CV} $\|.\|_\theta\leq C\|.\|_{2d+1}$ (on $A_{\rm reg}$) and thus it suffices to show that
$\mathfrak P^A_{2d+1}\big(\tilde\alpha(a-a\star_\theta^\alpha e_\lambda)\big)$ and 
$\mathfrak P^A_{2d+1}\big(\tilde\alpha(a-e_\lambda 
\star_\theta^\alpha a)\big)$ go to zero for all $a\in A_{\rm reg}$. For this, note that  for $
F_1,F_2\in\CB(\k^{2d},A)$,
we have
\begin{align*}
J^n(F_1\star_\theta F_2)= |2|_\k^{2d}
\int_{\k^{2d}\times\k^{2d}} \overline\Psi\big(2[Y,Z]\big)\,
\mu_0^{-2d-1}(Y)\,\mu_0^{-2d-1}(Z)\,J^n\big(\tau_{\theta Y}\big(J_\theta^{2d+1}F_1\big)\,
\tau_Z\big(J^{2d+1}F_2\big)\big)
\,dY\,dZ.
\end{align*}
Using the integral formula \eqref{PPP} applied to $J^n\big(\tau_{\theta Y}\big(J_\theta^{2d+1}F_1\big)\,
\tau_Z\big(J^{2d+1}F_2\big)\big)$, we deduce by commutativity of $J$ and $\tau$
and with $N=2d+1+n$:
\begin{align*}
&J^n(F_1\star_\theta F_2)(X)=|2|_\k^{6d}\int\overline\Psi\big(2[Y,Z]\big)
\overline\Psi\big(2[X,Y_1-Z_1+Y_2-Z_2]\big)
\overline\Psi\big(2[Y_1,Z_1]\big)\overline\Psi\big(2[Y_2,Z_2]\big)\\
&\times
\mu_0^{-2d-1}(Y)\,\mu_0^{-2d-1}(Z)\mu_0^n(Y_1-Z_1+Y_2-Z_2)
\mu_0^{-N}(Y_1-X)\mu_0^{-N}(Z_1-X)\mu^{-N}_0(Y_2-X)\mu^{-N}_0(Z_2-X)\\
&\times
\big(J^{N}J_\theta^{2d+1}F_1\big)(Y_1+\theta Y)\big(J^{N+2d+1}F_2\big)(Y_2+Z)
\,dYdZdY_1dZ_1dY_2dZ_2.
\end{align*}
On the other hand,  by Lemma \ref{IRL},  we have for all $F\in\CB(\k^{2d},A)$:
\begin{align*}
J^nF= |2|_\k^{2d}
\int_{\k^{2d}\times\k^{2d}} \overline\Psi\big(2[Y,Z]\big)\,
\mu_0^{-2d-1}(Y)\,\mu_0^{-2d-1}(Z)\,J^n\big(\tau_{\theta Y}\big(J_\theta^{2d+1}F\big)\big)
\,dY\,dZ.
\end{align*}
But the integral formula \eqref{PPP} generalizes    when $F_1\in\CB(\k^{2d},A)$
and $F_2\in\CB(\k^{2d},M(A))$. Applying it for $F_2=1$, 
 one deduces
\begin{align*}
&J^nF(X)=|2|_\k^{6d}\int \overline\Psi\big(2[Y,Z]\big)
\overline\Psi\big(2[X,Y_1-Z_1+Y_2-Z_2]\big)
\overline\Psi\big(2[Y_1,Z_1]\big)\overline\Psi\big(2[Y_2,Z_2]\big)\\
&\times
\mu_0^{-2d-1}(Y)\,\mu_0^{-2d-1}(Z)\mu_0^n(Y_1-Z_1+Y_2-Z_2)
\mu_0^{-N}(Y_1-X)\mu_0^{-N}(Z_1-X)\mu^{-N}_0(Y_2-X)\mu^{-N}_0(Z_2-X)\\
&\times
\big(J^{N}J_\theta^{2d+1}F_1\big)(Y_1+\theta Y)\,dYdZdY_1dZ_1dY_2dZ_2.
\end{align*}

These observations imply that for all $a\in A_{\rm reg}$, $n\in\N$ and with $N=n+2d+1$, we have
\begin{align*}
&J^n\big(\tilde\alpha(a)-\tilde\alpha(a\star_\theta^\alpha e_\lambda)\big)(X)=J^n\big(\tilde\alpha(a)-
\tilde\alpha(a)\star_\theta
\tilde\alpha(e_\lambda)\big)(X)\\
&=|2|_\k^{6d}\int dYdZdY_1dZ_1dY_2dZ_2\,
\overline\Psi\big(2[Y,Z]\big)
\overline\Psi\big(2[X,Y_1-Z_1+Y_2-Z_2]\big)
\overline\Psi\big(2[Y_1,Z_1]\big)\overline\Psi\big(2[Y_2,Z_2]\big)\\
&\times
\mu_0^{-2d-1}(Y)\,\mu_0^{-2d-1}(Z)\mu_0^n(Y_1-Z_1+Y_2-Z_2)
\mu_0^{-N}(Y_1-X)\mu_0^{-N}(Z_1-X)\mu^{-N}_0(Y_2-X)\mu^{-N}_0(Z_2-X)\\
&\times\Big(\big(J^{N}J_\theta^{2d+1}\tilde\alpha(a)\big)(Y_1+\theta Y)-
\big(J^{N}J_\theta^{2d+1}\tilde\alpha(a)\big)(Y_1+\theta Y)\big(J^{N+2d+1}\tilde\alpha(e_\lambda)\big)
(Y_2+Z)\Big).
\end{align*}
Using the Peetre inequality, the fact that the action $\alpha$ is isometric  and the 
(almost tautological) relation
$$
\big(J\tilde\alpha(a)\big)(X)=\alpha_X\big(J\tilde\alpha(a)(0)\big),
$$
we deduce 
\begin{align*}
&\mathfrak P^A_n\big(\tilde\alpha(a)-\tilde\alpha(a\star_\theta^\alpha e_\lambda)\big)\leq
\sup_{X\in\k^{2d}}
\int 
\mu_0^{-2d-1}(Y)\,\mu_0^{-2d-1}(Z)
\mu_0^{-2d-1}(Y_1-X)\\
&\times\mu_0^{-2d-1}(Z_1-X)\mu^{-2d-1}_0(Y_2-X)\mu^{-2d-1}_0(Z_2-X)
\Big\|\big(J^{n+2d+1}J_\theta^{2d+1}\tilde\alpha(a)\big)(Y_1+\theta Y-Y_2-Z)\\
&-
\big(J^{n+2d+1}J_\theta^{2d+1}\tilde\alpha(a)\big)(Y_1+\theta Y-Y_2-Z)\big(J^{n+4d+2}\tilde\alpha(e_\lambda)
\big)(0)\Big\|_A \,dYdZdY_1dZ_1dY_2dZ_2,
\end{align*}
Performing the translations $Y_j\mapsto Y_j+X$, $Z_j\mapsto Z_j+X$, we see that the integral
above does not depend on $X$. Performing the translation $Y_1\mapsto Y_1-\theta Y+Y_2+Z$
and using Fubini, we see that the integral above is of the form
\begin{align}
\label{AB}
\int_{\k^{2d}} \Phi(X) 
\big\|\big(J^{n+2d+1}J_\theta^{2d+1}\tilde\alpha(a)\big)(X)-
\big(J^{n+2d+1}J_\theta^{2d+1}\tilde\alpha(a)\big)(X)\big(J^{n+4d+2}\tilde\alpha(e_\lambda)
\big)(0)\big\|_A\,dX,
\end{align}
where $0\leq\Phi\in L^1(\k^{2d})$.

Next, we estimate the integral \eqref{AB} as a sum of two terms
\begin{align*}
I_1^\lambda&:=\int_{\k^{2d}} \Phi(X) 
\big\|\big(J^{n+2d+1}J_\theta^{2d+1}\tilde\alpha(a)\big)(X)-
\big(J^{n+2d+1}J_\theta^{2d+1}\tilde\alpha(a)\big)(X)\,e_\lambda\big\|_A\,dX,\\
I_2^\lambda&:=\int_{\k^{2d}} \Phi(X) 
\big\|
\big(J^{n+2d+1}J_\theta^{2d+1}\tilde\alpha(a)\big)(X)\big((J^{n+4d+2}-{\Id})\tilde\alpha(e_\lambda)
\big)(0)\big\|_A\,dX.
\end{align*}
Let now $C_n$ be the ball in $\k^{2d}\times\k^{2d}$ centered in $0$ and of radius $n$.
By absolute convergence (in norm) of the integral $I_1^\lambda$ 
and since $\|e_\lambda\|_A\leq\|\vf\|_1\|e'_\lambda\|_A\leq 1$, we deduce that
for each $\eps>0$ there exists $n_0\in\N$ (independent of $\lambda$) such that for all $n\geq n_0$ we have 
$$
\int_{\k^{2d}\setminus C_n} \Phi(X) 
\big\|\big(J^{n+2d+1}J_\theta^{2d+1}\tilde\alpha(a)\big)(X)-
\big(J^{n+2d+1}J_\theta^{2d+1}\tilde\alpha(a)\big)(X)\,e_\lambda\big\|_A\,dX\leq\eps.
$$
On the other hand, by a compactness argument and since 
$\{e_\lambda\}_{\lambda\in\Lambda}$ is an approximate unit, 
it easily follows that for any $n\in\N$, we have
$$
\lim_\lambda \int_{ C_n}\Phi(X) 
\big\|\big(J^{n+2d+1}J_\theta^{2d+1}\tilde\alpha(a)\big)(X)-
\big(J^{n+2d+1}J_\theta^{2d+1}\tilde\alpha(a)\big)(X)\,e_\lambda\big\|_A\,dX=0,
$$
hence $\lim_\lambda I_1^\lambda=0$. For the second bit, we first come back to the definition
of $e_\lambda\in A_{\rm reg}$ in terms of $e'_\lambda\in A$, to get
$$
(J^{N}-1)\tilde\alpha(e_\lambda)(0)=\int_{\k^{2d}}(J^{N}-1)\vf(Y) \,\alpha_Y(e'_\lambda)\,dY.
$$  
Since moreover $\int_{\k^{2d}}(J^{N}-1)\vf(Y) \,dY=0$, we get for any $b\in A$
$$
b\,(J^{N}-1)\tilde\alpha(e_\lambda)(0)=\int_{\k^{2d}}(J^{N}-1)\vf(Y) \,\big(b\,\alpha_Y(e'_\lambda)-b\big)\,dY.
$$ 
Applying this formula to $b=\big(J^{n+2d+1}J_\theta^{2d+1}\tilde\alpha(a)\big)(X)$, we get the bound
\begin{align*}
I_2^\lambda&\leq\int_{\k^{2d} \times \k^{2d}}\Phi(X)\big|
(J^{N}-1)\vf(Y)\big|\\
&\qquad\qquad\qquad\times
\big\|\big(J^{n+2d+1}J_\theta^{2d+1}\tilde\alpha(a)\big)(X)\,
\alpha_Y(e'_\lambda)-\big(J^{n+2d+1}J_\theta^{2d+1}\tilde\alpha(a)\big)(X)\big\|_A\,dX\,dY.
\end{align*}
Using one more time the isometricity of the action, we finally get,
\begin{align*}
I_2^\lambda&\leq\int_{\k^{2d} \times \k^{2d}}\Phi(X)\big|
(J^{N}-1)\vf(Y)\big|\\
&\qquad\qquad\qquad\times
\big\|\big(J^{n+2d+1}J_\theta^{2d+1}\tilde\alpha(a)\big)(X-Y)\,
e'_\lambda-\big(J^{n+2d+1}J_\theta^{2d+1}\tilde\alpha(a)\big)(X-Y)\big\|_A\,dX\,dY.
\end{align*}
Noting that   $(J^{N}-1)\vf\in\CS(\k^{2d})$, which can be approximated in $\CD(\k^{2d})$,
we conclude that $\lim_\lambda I_2^\lambda=0$ with the same compactness argument than the
one we used for $I_1^\lambda$.
 The case of  $a-e_\lambda \star_\theta^\alpha a$ is entirely similar.
 \end{proof}
In particular, we deduce that if $A$ is $\sigma$-unital, so does $A_\theta$ and, thanks to 
Theorem \ref{INV}, we get that $A$ is $\sigma$-unital if and only if $A_\theta$ is. 

\begin{rmk}
\label{BAUF}
The proof of the proposition above gives the existence of a bounded approximate
unit for the Fr\'echet algebra $(A_{\rm reg},\star_\theta^\alpha)$, in the sense that for all $n\in\N$,
$\sup_{\lambda\in\Lambda}\|e_\lambda\|_n<\infty$  and for all $a\in A_{\rm reg}$,
$\lim_{\lambda}\|a\star_\theta e_\lambda-a\|_n=\lim_{\lambda}\|e_\lambda\star_\theta a-a\|_n=0$.
\end{rmk}

We next study the question of compatibility of the deformation with ideals and morphisms.
The following two results   follow from  minor  modifications of the
similar statements in \cite{Ri}. (See \cite[Proposition 3.8]{K1} for an alternative proof of
Proposition \ref{InjSurj}.)
\begin{prop}
\label{InjSurj}
Let $(A,\alpha)$ and $(B,\beta)$ be two $C^*$-algebras endowed with  
continuous actions of $\k^{2d}$ and let  $T:A\to B$ be a $*$-homomorphism
which intertwines the actions $\alpha$ and $\beta$. Then,
 $T$ maps $A_{\rm reg}$ to $B_{\rm reg}$ and extends to a continuous homomorphism
$T_\theta:A_\theta\to B_\theta$ which intertwines the actions $\alpha_\theta$ and $\beta_\theta$.  If
 moreover $T$ is injective (respectively surjective) then $T_\theta$ is injective (respectively  surjective) too.
\end{prop}
\begin{proof}
 Let $a\in A_{\rm reg}$. Then from the equality (in $C_{u}(\k^{2d},A)$) $\tilde\beta(T(a))=\Id\otimes T
\big( \tilde\alpha(a)\big)$
together with the fact that $*$-homomorphisms are norm decreasing, we  get $\|T(a)\|_n\leq\|a\|_n$,
which implies that $T$ maps $A_{\rm reg}$ to $B_{\rm reg}$. From the absolute convergence of the integral formula 
\eqref{OM} for $\star_\theta$ at the level of $A_{\rm reg}$ we deduce that $T$ is a continuous $*$-homomorphism
from $(A_{\rm reg},\star_\theta^\alpha)$ to $(B_{\rm reg},\star_\theta^\alpha)$. 
From arguments identical to those of 
\cite[Theorem 5.7]{Ri} (using the $C^*$-module approach to the deformed $C^*$-norm as explained in  
Proposition \ref{CMA}), we see that $T$ is continuous for the $C^*$-norms, hence it extends to the 
completions. $T_\theta$, the extension of $T$, also intertwines the actions because the actions 
 have not changed. 
That $T_\theta$ is injective (respectively  surjective) when $T$ is injective (respectively surjective)
can be proven exactly as in   \cite[Proposition 5.8]{Ri}.
\end{proof}

\begin{prop}
Let $(A,\alpha)$ be a $C^*$-algebra endowed with a
continuous   action of $\k^{2d}$ and let  $I$ be an $\alpha$-invariant (essential) ideal
of $A$. Then $I_\theta$ is an (essential) ideal of $A_\theta$. 
\end{prop}
\begin{proof}
This is exactly the arguments of \cite[Proposition 5.9]{Ri} except the fact that we need to show
that if $I$ is an $\alpha$-invariant (essential) ideal of $A$ then, $I_{\rm reg}$ is an ideal of $(A_{\rm reg},
\star_\theta^\alpha)$.
But this fact again follows from the absolute  convergence of the integral formula 
\eqref{OM}.
\end{proof}

Now we come to a very important point, namely that the deformation can be performed 
in stages.  

\begin{thm}
\label{INV}
Let $\theta,\theta'\in\CO_\k$. Then, $(A_{\theta})_{\theta'}\simeq A_{\theta+\theta'}$
 and moreover $(A_{\theta})_{\rm reg}=A_{\rm reg}$.
\end{thm}
\begin{proof}
That $(A_{\theta})_{\theta'}\simeq A_{\theta+\theta'}$ follows from \cite[Lemma 3.5]{K1}
or \cite[Theorem 3.10]{NT} and it remains to prove that $(A_{\theta})_{\rm reg}=A_{\rm reg}$.
The first step is to show that $A_{\rm reg}\subset (A_{\theta})_{\rm reg}$ with dense inclusion. 
By construction,  $A_{\rm reg}\subset A_{\theta}$, so that
it makes sense to evaluate the seminorms $\|.\|_n^{A_\theta}$ (i.e$.$ those giving the Fr\'echet 
topology of $(A_{\theta})_{\rm reg}$) on $A_{\rm reg}$:
$$
\|a\|_n^{A_\theta}=\mathfrak P_n^{A_\theta}\big(\tilde\alpha(a)\big)=\sup_{X\in\k^{2d}}
\|J^n\tilde\alpha(a)(X)\|_\theta\leq C\!\!\sup_{X\in\k^{2d}}\|J^n\tilde\alpha(a)(X)\|_{2d+1}^A=C\!\!
\sup_{X\in\k^{2d}}\mathfrak P_{2d+1}^A\big(J^n\tilde\alpha(a)(X)\big),
$$
but it is easy to see that the later expression coincides with 
$\mathfrak P_{n+2d+1}^A\big(\tilde\alpha(a)\big)=\|a\|_{n+2d+1}^A$,
showing that  $A_{\rm reg}\subset (A_{\theta})_{\rm reg}$.  That $A_{\rm reg}$ is dense in  
$(A_{\theta})_{\rm reg}$, 
 follows from the Dixmier-Malliavin Theorem for
 general locally compact groups \cite[Theorem 4.16]{Meyer}. Indeed, let $a\in (A_\theta)_{\rm reg}$, 
 $\eps>0$ and $n\in\N$.
By Proposition \ref{DAA}, $(A_\theta)^\infty\subset (A_\theta)_{\rm reg}$ densely so that there exists $b\in
(A_\theta)^\infty$ with $\|a-b\|_n^{A_\theta}\leq\eps/2$.  Now, by  \cite[Theorem 4.16]{Meyer},
there exists $b_1,\dots,b_k\in A_\theta$ and $\vf_1,\dots,\vf_k\in\CD(\k^{2d})\subset\CS(\k^{2d})$
such that $b=\sum_{j=1}^k\alpha_{\vf_j}(b_j)$. But by construction $A_{\rm reg}$ is dense in $A_\theta$
so that there exists $c_1,\dots,c_k\in A_{\rm reg}$ with $\|b_j-c_j\|_\theta\leq \eps/(2k\|J^n \vf_j\|_1)$.
Setting $c:=\sum_{j=1}^k\alpha_{\vf_j}(c_j)$, we finally deduce
\begin{align*}
\|a-c\|_n^{A_\theta}&\leq\|a-b\|_n^{A_\theta}+\|b-c\|_n^{A_\theta}\leq \frac\eps2+\sum_{j=1}^k
\|\alpha_{\vf_j}(b_j-c_j)\|_n^{A_\theta}.
\end{align*}
 Now, from the (already used) relation $J^n\tilde\alpha\big(\alpha_\vf(a)\big)=
 \tilde\alpha\big(\alpha_{J^n\vf}(a)\big
 )$, valid for $a\in A_\theta$ and $\vf\in\CS(\k^{2d})$ and since the action $\alpha$ is still isometric
 on $A_\theta$, we deduce that 
 $$
 \|\alpha_{\vf_j}(b_j-c_j)\|_n^{A_\theta}=\|\alpha_{J^n\vf_j}(b_j-c_j)\|_\theta\leq \|J^n\vf_j\|_1\|b_j-c_j\|_\theta,
 $$ 
 and thus 
 $$
 \|a-c\|_n^{A_\theta}\leq  \frac\eps2+\sum_{j=1}^k
\|J^n\vf_j\|_1 \|b_j-c_j\|_n^{A_\theta}\leq\eps,
$$
 as needed. 
 The reversed inclusion  follows from the first part: we have seen that 
$A_{\rm reg}\subset (A_\theta)_{\rm reg}$ for any $C^*$-algebra $A$ endowed with a  continuous 
action 
 of $\k^{2d}$. Applying this to the deformed $C^*$-algebra $A_\theta$, which still carries
  a  continuous action of $\k^{2d}$, we deduce that for any $\theta'\in \CO_\k$, we have
  $(A_\theta)_{\rm reg}\subset ((A_\theta)_{\theta'})_{\rm reg}$ but since $((A_\theta)_{\theta'})_{\rm reg}
  =(A_{\theta+\theta'})_{\rm reg}$,
  we deduce for $\theta'=-\theta$ that $(A_\theta)_{\rm reg}\subset A_{\rm reg}$, which completes the proof.
\end{proof}
 
 With the help of the above theorem, we can use the same proof than \cite[Theorem 7.7]{Ri}
 to get that the deformation maps equivariant short exact sequences to  short exact sequences.
 Alternatively, one can use \cite[Theorem 3.9]{K1}.
\begin{thm}
\label{ESES}
Let $I$ be an $\alpha$-invariant ideal of $A$ and let  $Q:=A/I$ endowed with the quotient 
action. Then the equivariant short exact sequence
$$
0\to I\to A\to Q\to 0,
$$
gives rise to a short exact sequence of deformed algebras:
$$
0\to I_\theta\to A_\theta\to Q_\theta\to 0.
$$
\end{thm}

Our last result concerns continuity of the field of deformed $C^*$-algebras $(A_\theta)_{\theta\in\CO_\k}$.
Here, the continuity structure refers to the $*$-subalgebra $A_{\rm reg}$, viewed as 
a subspace of constant sections. For the question of continuity, the twisted crossed product approach does
not seem to be especially appropriate, contrary to the methods developed in \cite{Ri}. Note however that
due to particularity of non-Archimedean analysis, we are forced to consider a restricted range of 
parameters.
\begin{thm}
Let $\gamma\in\CO_\k$ fixed. Then, the field of deformed $C^*$-algebras 
$(A_{\gamma\theta^2})_{\theta\in \CO_\k}$ is continuous.
\end{thm} 
\begin{proof}
Our proof mimics \cite[Chapter 8]{Ri}, where continuity is obtained from combination of lower
and upper semicontinuity.
By an immediate adaptation of the arguments given in \cite[page 55]{Ri}, lower semicontinuity
of the field  $(A_\theta)_{\theta\in \CO_\k}$
will follow if for all $a\in A_{\rm reg}$, all $f_1\in \CD(\k^{2d},A)$ and $f_2\in\CS(\k^{2d},A)$ with
$\CG(f_2)\in \CD(\k^{2d},A)$, we have
$$
\big\|\langle f_1,\tilde\alpha(a)\star_\theta f_2\rangle_A-
\langle f_1,\tilde\alpha(a)\star_{\theta'} f_2\rangle_A\big\|_A\to0,\quad \theta\to\theta'.
$$
But this easily follows from the strong continuity of $\alpha$ and a compactness argument once one
has realized that
\begin{align*}
&\langle f_1,\tilde\alpha(a)\star_\theta f_2\rangle_A-
\langle f_1,\tilde\alpha(a)\star_{\theta'} f_2\rangle_A\\
&\qquad=
|2|_\k^d\int_{\k^{2d}\times\k^{2d}}\overline\Psi(2[X,Y])\,
f_1(X) \, \alpha_{X}(\alpha_{\theta Y}(a) - \alpha_{\theta'Y}(a))   \,\CG(f_2)(Y) dX dY.
\end{align*}
In particular, the field $(A_{\gamma\theta^2})_{\theta\in \CO_\k}$ is lower-semicontinuous
for any $\gamma\in \CO_\k$.

Upper semicontinuity relies on   \cite[Proposition 1.2]{Ri-1}. To be able to use this result, 
we must let the action $\alpha$
variate.
 So, fix  $\gamma\in\CO_\k$ and define a new action $\alpha^\gamma$ of $\k^{2d}$ on $A$
 by $\alpha^\gamma_X(a) :=\alpha_{\gamma X}(a)$.
It is clear that $\alpha^\gamma$ is still  continuous. To do not get confused, we need extra notations.
We now let  $A_{\rm reg}^\alpha$ to be our dense Fr\'echet subspace of $A$ as given in \eqref{A0} for a
given action $\alpha$. 
Accordingly, we  denote by $\|.\|_n^\alpha$ to be seminorms on $A_{\rm reg}^\alpha$ as defined in 
\eqref{SNn} and by $A_\theta^\alpha$ the deformed $C^*$-algebra. We first observe that 
$A_{\rm reg}^\alpha\subset A_{\rm reg}^{\alpha^\gamma}$ with continuous and dense  inclusion. Indeed,
for  $a\in A_{\rm reg}^\alpha$, using  the notation $D_\gamma F(X):=F(\gamma X)$, $F\in\CB(\k^{2d},A)$, we have
$\tilde\alpha^\gamma(a)=D_\gamma\tilde\alpha(a)$ and thus we deduce from 
 Lemma \ref{DILL} (iii):
$$
\|a\|_n^{\alpha^\gamma}=\mathfrak P_n^A\big(\tilde\alpha^\gamma(a)\big)
\leq\|\mu_0^{-2d-1}\|_1^2 \,
\mathfrak P_{n+2d+1}^A\big(\tilde{\alpha}(a)\big)=\|\mu_0^{-2d-1}\|_1^2\,\|a\|_{n+2d+1}^{\alpha},
$$
and the continuity follows. For the density, one observes  that the spaces  of smooth vectors 
(in the sense of Bruhat) for the actions $\alpha$ and $\alpha^\gamma$ (for $\gamma\ne 0$) coincide since
 $\tilde\alpha(a)$ is locally constant if and only if
 $\tilde\alpha^\gamma(a)$ does. Thus  $A^\infty$ is dense in $A_{\rm reg}^{\alpha^\gamma}$ and $A^\infty$ is 
 contained
 in $A_{\rm reg}^\alpha$ so $A_{\rm reg}^\alpha$ is dense in $A_{\rm reg}^{\alpha^\gamma}$. 
 Next we compare the deformed $C^*$-algebras  $A_\theta^{\alpha^\gamma}$
 and $A_{\gamma^2\theta}^\alpha$.
 Let $F\in\CB(\k^{2d})$ and $f\in\CS(\k^{2d})$. Undoing partially the oscillatory trick in Eq. \eqref{Star}, we get after some rearrangements:
 $$
 F\star_\theta f(X)=|2|^{d}_\k\int_{\k^{2d}}\overline\Psi(2[X,Y])\, F(X+\theta Y)\,\big(\CG f\big)(Y)\, dY.
 $$
  From this and the scaling relation  $\CG D_\gamma(f)=|\gamma|_\k^{-2d}D_\gamma\CG (f)$, 
  $\gamma\in\CO_\k \backslash \{0 \}$, 
  we deduce that $D_\gamma(F)\star_\theta D_\gamma(f)=D_\gamma\big(F\star_{\gamma^2\theta}f\big)$ 
  in $\CS(\k^{2d},A)$.
 Introducing the unitary operator $U_\gamma:=|\gamma|_\k^{d} D_\gamma$ on the pre-$C^*$-module
  $\CS(\k^{2d},A)$, the previous relation entails that for any 
 $a\in A_{\rm reg}^\alpha$, we have:
 \begin{align}
 \label{UNIT}
 U_\gamma^*\, L_\theta\big(\tilde\alpha^\gamma(a)\big)\,U_\gamma=
 L_{\gamma^2\theta}\big(\tilde\alpha(a)\big),
\end{align}
where the equality holds in the $C^*$-algebra of adjointable bounded $A$-linear 
endomorphisms of the pre-$C^*$-module $\CS(\k^{2d},A)$. The above relation also gives 
$L_{\gamma^2\theta}\big(\tilde\alpha(a\star_\theta^{\alpha^\gamma}b)\big)=
L_{\gamma^2\theta}\big(\tilde\alpha(a\star_{\gamma^2\theta}^{\alpha}b)\big)$, for $a,b\in A_{\rm reg}^\alpha$,
and thus
$$
0=\big\|L_{\gamma^2\theta}\big(\tilde\alpha(a\star_\theta^{\alpha^\gamma}b)\big)-
L_{\gamma^2\theta}\big(\tilde\alpha(a\star_{\gamma^2\theta}^{\alpha}b)\big)\big\|=
\big\|L_{\gamma^2\theta}\big(\tilde\alpha(a\star_\theta^{\alpha^\gamma}b-
a\star_{\gamma^2\theta}^{\alpha}b)\big)\big\|=\|a\star_\theta^{\alpha^\gamma}b-
a\star_{\gamma^2\theta}^{\alpha}b\|_{\gamma^2\theta}.
$$
Hence, $a\star_\theta^{\alpha^\gamma}b=
a\star_{\gamma^2\theta}^{\alpha}b$ in $A^\alpha_{\gamma^2\theta}$ but since 
$a\star_{\gamma^2\theta}^{\alpha}b\in A_{\rm reg}^\alpha$ 
(a priori, $a\star_{\gamma^2\theta}^{\alpha}b\in A_{\rm reg}^{\alpha^\gamma}$) the equality takes 
place
within $A_{\rm reg}^\alpha$. But the relation \eqref{UNIT} also shows 
that
$$
\|a\|_\theta^{\alpha^\gamma}=\big\|L_\theta\big(\tilde\alpha^\gamma(a)\big)\big\|
=\big\| L_{\gamma^2\theta}\big(\tilde\alpha(a)\big)\big\|=\|a\|_{\gamma^2\theta}^\alpha\,,\quad\forall a\in 
A_{\rm reg}^\alpha.
$$
Since $A_{\rm reg}^\alpha$ is dense both in $A_\theta^{\alpha^\gamma}$ and in $A_{\gamma^2\theta}^\alpha$,
we deduce that $A_\theta^{\alpha^\gamma}=A_{\gamma^2\theta}^\alpha$. Hence (inverting the roles
of $\theta\in\CO_\k$ and  of $\gamma\in\CO_\k$), it suffices to show that the field
$(A_\theta^{\alpha^\gamma})_{\gamma\in\CO_\k}$ is upper-semicontinuous. To this aim, consider
on the $C^*$-algebra $C_0(\CO_\k,A)$ the action of $\k^{2d}$ given by:
$$
\beta_X(\Phi)(\gamma):=\alpha_{\gamma X}\big(\Phi(\gamma)\big).
$$
The space $\CO_\k$ being compact, one easily sees that $\beta$ is  continuous. For fixed
$\gamma\in\CO_\k$, let $e_\gamma:C_0(\CO_\k,A)\to A$ be the evaluation map at $\gamma$
and let $C_0^\gamma(\CO_\k,A)$ the (norm closed) ideal of elements in $C_0(\CO_\k,A)$ vanishing
at $\gamma$. The associated short exact sequence $0 \to C_0^\gamma(\CO_\k,A)\to C_0(\CO_\k,A)
\to A\to 0$ being equivariant for $\beta$ on $C_0^\gamma(\CO_\k,A)$ and on
$C_0(\CO_\k,A)$, and for $\alpha^\gamma$ on $A$ (since $e_\gamma$ intertwines $\beta$ and
 $\alpha^\gamma$), we deduce from Theorem \ref{ESES} that we have a short exact sequence 
 of deformed $C^*$-algebras:
 $$
 0 \to C_0^\gamma(\CO_\k,A)_\theta^\beta\to C_0(\CO_\k,A)_\theta^\beta
\to A_\theta^{\alpha^\gamma}\to 0.
$$
Moreover, as $C_0(\CO_\k)$ (seen as a subalgebra of $M( C_0(\CO_\k,A))$) is left 
invariant by the action $\beta$, it is its own space of regular elements and by Proposition \ref{prop1}
$\Phi\star_\theta^\beta \eta=\Phi\,\eta=\eta\star_\theta^\beta\Phi$ for all $\Phi\in C_0(\CO_\k,A)_{\rm reg}^\beta$
and all $\eta\in C_0(\CO_\k)$. Hence, $C_0(\CO_\k)$  may also be viewed as a subalgebra of 
$M( C_0(\CO_\k,A)_\theta^\beta)$. Let then $C_0^\gamma(\CO_\k)$ be the norm closed ideal in
$C_0(\CO_\k)$ of elements vanishing at $\gamma$ and let $\overline{C_0(\CO_\k,A)_\theta^\beta
C_0^\gamma(\CO_\k)}$ be the norm closure in the deformed $C^*$-algebra 
$C_0(\CO_\k,A)_\theta^\beta$ of the linear span of products. Then by \cite[Proposition 1.2]{Ri-1}
the field of $C^*$-algebra 
$$
\Big(C_0(\CO_\k,A)_\theta^\beta\Big/\,\overline{C_0(\CO_\k,A)_\theta^\beta
C_0^\gamma(\CO_\k)}\Big)_{\gamma\in\CO_\k},
$$
is upper-semicontinuous. 
But since $A_\theta^{\alpha^\gamma}\simeq C_0(\CO_\k,A)_\theta^\beta/ 
C_0^\gamma(\CO_\k,A)_\theta^\beta$, it suffices to show that $C_0^\gamma(\CO_\k,A)_\theta^\beta$
coincides with
$\overline{C_0(\CO_\k,A)_\theta^\beta C_0^\gamma(\CO_\k)}$. On the one hand, we have
$C_0(\CO_\k,A)C_0^\gamma(\CO_\k)\subset C_0^\gamma(\CO_\k,A)$ and thus at the level of the regular 
vectors $\big(C_0(\CO_\k,A)C_0^\gamma(\CO_\k)\big)_{\rm reg}^\beta \subset C_0^\gamma(\CO_\k,A)
_{\rm reg}^\beta
\subset C_0^\gamma(\CO_\k,A)_\theta^\beta$. But 
$\big(C_0(\CO_\k,A)C_0^\gamma(\CO_\k)\big)_{\rm reg}^\beta=C_0(\CO_\k,A)_{\rm reg}^\beta 
C_0^\gamma(\CO_\k)
=C_0(\CO_\k,A)_{\rm reg}^\beta \star_\theta^\beta C_0^\gamma(\CO_\k)$ and thus 
$\overline{C_0(\CO_\k,A)_\theta^\beta
 C_0^\gamma(\CO_\k)}$ $\subset C_0^\gamma(\CO_\k,A)_\theta^\beta$.
On the other hand since $\overline{C_0(\CO_\k,A) C_0^\gamma(\CO_\k)}=C_0^\gamma(\CO_\k,A)$, 
we have by the Cohen Factorization Theorem
$C_0^\gamma (\CO_\k,A)=C_0(\CO_\k,A) C_0^\gamma(\CO_\k)$ (see for instance
 \cite[Theorem 32.22]{HR2}). Hence, any element $\Phi\in C_0^\gamma(\CO_\k,A)_{\rm reg}^\beta$
 can be written as $\Phi=\Xi\eta$ with $\Xi\in C_0(\CO_\k,A)$ and $\eta\in C_0^\gamma(\CO_\k)$.
 For $\vf\in \CS(\k^{2d})$, we have $\beta_\vf(\Phi)\in C_0^\gamma(\CO_\k,A)_{\rm reg}^\beta$ and 
 $\beta_\vf(\Phi)=\beta_\vf(\Xi)\eta=\beta_\vf(\Xi)\star_\theta^\beta\eta\in C_0(\CO_\k,A)_{\rm reg}^\beta 
 \star_\theta^\beta C_0^\gamma(\CO_\k)$. But by Proposition \ref{BAU} $\Phi$ is approximated
 in $C_0^\gamma(\CO_\k,A)_\theta^\beta$ by elements of the form  $\beta_\vf(\Phi)$. Hence
 $C_0^\gamma(\CO_\k,A)_{\rm reg}^\beta\subset \overline{C_0(\CO_\k,A)_\theta^\beta
 C_0^\gamma(\CO_\k)}$ and thus  $C_0^\gamma(\CO_\k,A)_\theta^\beta
 \subset \overline{C_0(\CO_\k,A)_\theta^\beta
 C_0^\gamma(\CO_\k)}$, concluding the proof.
 \end{proof}

\end{document}